\documentclass[journal,twoside,web]{ieeecolor}
\usepackage{generic}
\usepackage{hyperref}
\usepackage{tabularx}
\usepackage{booktabs}       
\usepackage{tikz} 
\usepackage{pifont}
\usetikzlibrary{positioning, fit, arrows.meta, shapes, calc}

\usepackage{textcomp}
\def\BibTeX{{\rm B\kern-.05em{\sc i\kern-.025em b}\kern-.08em
    T\kern-.1667em\lower.7ex\hbox{E}\kern-.125emX}}
\usepackage{amsthm,amsmath,amssymb}
\usepackage{threeparttable} 

\usepackage{ifpdf}
\usepackage{amsfonts}
\usepackage{mathrsfs}
\usepackage{graphicx}
\usepackage{subfigure}
\usepackage{epstopdf}
\usepackage{lineno}
\usepackage{cite}
\usepackage{enumerate}
\usepackage{algorithm, algorithmic}
\usepackage{bm}
\usepackage{tikz}

\newtheorem{theorem}{Theorem}

\newtheorem{lemma}{Lemma}     
\newtheorem{remark}{Remark}

\newtheorem{assumption}{Assumption}

\begin{document}
\title{Wasserstein Robust Performative Prediction via Lagrangian Relaxation}
\author{Siyi Wang, Zifan Wang, Karl H. Johansson, \IEEEmembership{Fellow, IEEE}
\thanks{This work was supported by the Swedish Research Council Distinguished Professor Grant 2017-01078, Knut and Alice Wallenberg Foundation, Wallenberg Scholar Grant, and Swedish Strategic Research Foundation SUCCESS Grant FUS21-0026.}
\thanks{Siyi~Wang, Zifan~Wang, and Karl H. Johansson are with the Department of Decision and Control Systems, School of Electrical Engineering and Computer Science, KTH Royal Institute of Technology,
10044  Stockholm, Sweden, e-mail: \{siyiw, zifanw, kallej\}@kth.se. They are also affiliated with Digital Futures.} }
\maketitle

\begin{abstract}
In machine learning, predictive models are trained on historical data. Their deployment may incentivize agents to strategically adapt their behavior, thereby inducing a model-dependent distribution shift. 
This phenomenon is known as performativity.
This paper develops a Wasserstein distributionally robust framework for performative prediction, where the predictive model only has access to limited data. 
Using these data,
we construct an ambiguity set centered on the empirical distribution, and optimize the predictive model against the worst-case distribution. 
Furthermore, we reformulate the objective as a tractable min-max optimization problem via Lagrangian relaxation, and allow the penalty to depend on the prediction model.
Based on this, we develop distributionally robust repeated risk minimization  (DR-RRM) and repeated gradient descent (DR-RGD) algorithms to iteratively find a performative stable point amid distributional shifts and model retraining. We theoretically show that both algorithms converge to a stable point linearly under standard regularity conditions. When accounting for approximation errors in the optimization problems, both algorithms converge to a neighborhood of the stable point. Additionally, we establish theoretical bounds on the suboptimality gap between the stable point and the global performative optimum. 
 Finally, numerical simulations of a dynamic credit scoring problem demonstrate the efficacy of the method.
\end{abstract}

\begin{IEEEkeywords} 
Decision-dependent distributions, distributionally robust optimization, risk minimization 
\end{IEEEkeywords}

\section{Introduction}\label{sec:introduction}

In machine learning, training predictive models is typically formulated as the problem of finding parameters that minimize expected risk \cite{spall2005introduction}.
This approach has been widely applied in fields such as recommendation systems \cite{tang2023zeroth}, automated trading \cite{kraft2023stochastic}, and multi-robotic systems \cite{ma2026stochastic}. 
Standard approaches typically treat the data distribution as either fixed or exogenously nonstationary \cite{besbes2015non}. 
In practice, deploying a predictive model may incentivize agents to adapt their behavior, thereby shifting the underlying data distribution. 
For instance, in credit scoring, applicants may strategically adjust their financial profiles to align with the scoring rules, thereby improving their chances of approval. 
Similarly, in traffic routing, predicting congestion and suggesting optimal routes can alter driver behavior, creating new traffic patterns \cite{keimer2018information}. 
This model-induced distribution shift is formalized as performativity \cite{perdomo2020performative,hardt2025performative}, and has been studied in data-driven control \cite{bianchin2023online}, multiplayer games \cite{narang2023multiplayer,wang2023network,le2025learning}, and reinforcement learning~\cite{bertrand2023stability}. 

In performative prediction, deploying a model changes the underlying population distribution. To adapt to this shift, the model is retrained using data collected after deployment.
This gives rise to a dynamic feedback loop involving model deployment, distribution evolution, and model retraining. In this setting, two solution concepts have been widely studied: the performative optimal point \cite{miller2021outside}, which minimizes the performative risk, and the performative stable point \cite{perdomo2020performative}, where the model is optimal with respect to the population distribution it induces.

In practical retraining, the exact population distribution is rarely available. The learner typically observes only a limited number of samples, which may be subject to measurement noise or environmental shifts. This induces an additional distributional uncertainty beyond performativity. 
Moreover, training models on empirical distributions constructed from finite samples often leads to overfitting and suboptimal generalization. 
To mitigate this risk, distributionally robust optimization (DRO) provides a principled framework for addressing such distributional uncertainty \cite{mohajerin2018data,blanchet2019quantifying,kuhn2025distributionally}. Instead of minimizing the expected loss under a nominal or empirical distribution, DRO minimizes the worst-case expected loss over an ambiguity set of plausible distributions \cite{kuhn2025distributionally}.  
The ambiguity set defines a class of candidate distributions centered at the empirical data, and is typically constructed using either $\phi$-divergence \cite{duchi2021statistics} or the Wasserstein metric \cite{mohajerin2018data,shafieezadeh2019regularization}. 
The latter has gained traction due to its capability to capture distribution geometry and handle shifted support. 
Additionally, leveraging duality results, many DRO formulations can be recast into computationally tractable convex programs \cite{mohajerin2018data}.
Given these advantages, DRO has been extensively explored for control systems \cite{shafieezadeh2018wasserstein,coulson2021distributionally,taskesen2023distributionally,mcallister2024distributionally,brouillon2025distributionally} and machine learning \cite{kuhn2019wasserstein,liu2022distributionally}.

\begin{table*}
\centering
\caption{Comparison of frameworks and theoretical results.}
\label{tab:comparison}
\begin{tabular}{lcccccc}
\hline
Ref. & Solution point & Distance metric & Adjustable ambiguity size  & RRM-type Algorithm & RGD-type Algorithm  \\ 
\hline
\cite{luo2020distributionally} & Optimal  &   Wasserstein and others &    \ding{51}  &  \ding{55} & --             \\
\cite{qu2025decision}  &  Optimal  &    Wasserstein & \ding{51}  & \ding{55} & -- \\
\cite{noyan2020distributionally} &  Optimal  & Wasserstein &  \ding{55} & \ding{55} & -- \\
\cite{fonseca2023decision}  &  Optimal  &   Wasserstein &  \ding{55} & \ding{55} & -- \\
\cite{xue2024distributionally} &  Stable  & KL divergence &\ding{55}  & \ding{55} & \ding{55} \\
\cite{jia2025distributionally} & Stable  & Wasserstein & \ding{55} & \ding{51} & \ding{55} \\  \hline
{\textbf{This work}} &  Stable  & Wasserstein & \ding{51} & \ding{51} & \ding{51} \\ \hline
\end{tabular}%
\smallskip
\flushleft
\textbf{RRM} assumes access to the explicit form of the function, whereas \textbf{RGD} relies on gradient estimates.
\vspace{-2em}
\end{table*}

\subsection{Our Contributions}

In this paper, we investigate distributionally robust performative prediction under limited data access, and provide both algorithmic solutions and theoretical analyses. 
Unlike standard performative prediction \cite{perdomo2020performative}, which optimizes the risk over the induced true distribution, we optimize the distributionally robust performative risk, which is defined as the worst-case expected loss over a Wasserstein ambiguity set centered at the empirical distribution.   
To ensure computational tractability, we adopt a Lagrangian relaxation with a decision-dependent penalty. 
This corresponds to an adjustable ambiguity set and subsumes the standard fixed formulation as a special case.
Leveraging strong duality \cite{sinha2017certifying}, we reformulate the DRO into a minmax optimization over a worst-case distribution. 
Building on concepts from \cite{perdomo2020performative}, we define the robust performative stable point and the robust performative optimum. Solving the performative optimum is generally difficult, as both the distribution and the loss depend on the decision variable. 
Thus, we focus on finding a robust performative stable point, and develop distributionally robust repeated risk minimization (DR-RRM) and repeated gradient descent (DR-RGD) algorithms to solve it. Both algorithms operate iteratively: at each step, the learner samples from the current distribution and identifies adversarial points via a robust surrogate objective. Based on the worst-case distribution induced by these adversarial points, the model parameters are updated either by explicitly solving the risk minimization problem (DR-RRM), or by constructing a gradient estimate to take a descent step (DR-RGD). We show that both algorithms converge under standard regularity conditions such as smoothness and Lipschitz continuity. 
Specifically, our analysis accounts for the approximation error introduced by finding the worst-case distribution. We demonstrate that in the presence of the approximation error, the iterations converge to a bounded neighborhood of the equilibrium. In the ideal case, when the error is zero, the algorithms converge to a unique robust performative stable point. Finally, we bound both the distance and the performative risk gap between the performative stable point and the optimum.

Our contributions can be summarized as follows:
(1) Framework: We establish a distributionally robust performative prediction framework and convert the DRO objective into a tractable minmax problem via strong duality.
(2) Algorithms: We develop DR-RRM and DR-RGD algorithms to iteratively solve the robust performative stable point (Algorithms~\ref{alg:risk minimization} and \ref{alg:gradient descent}).
(3) Convergence analysis: Under standard regularity assumptions, we prove that both algorithms converge to a unique stable point in the exact setting, and to its bounded neighborhood in the presence of the approximation error (Theorems~\ref{theo:risk minimization} and \ref{theo:gradient descent}).
(4) Suboptimality guarantee: We derive theoretical bounds on the distance and risk gap between the accessible performative stable point and the global performative optimum (Theorem~\ref{theorem:suboptimality}).

\subsection{Related literature}

Distributionally robust performative prediction, also studied under the name decision-dependent distributionally robust optimization, has been investigated in \cite{luo2020distributionally,qu2025decision,noyan2020distributionally,fonseca2023decision,xue2024distributionally,jia2025distributionally}.
The main distinction among these works lies in their solution concept and algorithmic objective. Specifically, \cite{luo2020distributionally,qu2025decision,noyan2020distributionally,fonseca2023decision} focus on performative optimality, whereas our paper studies robust performative stability of the repeated retraining dynamics.
Closest to our work are \cite{xue2024distributionally,jia2025distributionally}, which also investigate distributionally robust formulations for performative stability.
Specifically, \cite{xue2024distributionally} considers ambiguity sets based on the Kullback--Leibler (KL) divergence. In contrast, we study Wasserstein ambiguity sets, which capture geometric perturbations of the data distribution and allow support shifts. 
This difference leads to a different formulation, and the techniques developed in \cite{xue2024distributionally} are not directly applicable to our setting.
The work \cite{jia2025distributionally} studies performative stability under Wasserstein ambiguity, and proposes a repeated risk minimization algorithm with convergence guarantees.
However, their method assumes access to an exact minimization oracle at each iteration, which can be computationally intractable in certain settings. 
To address this, we develop a repeated gradient-descent algorithm that relaxes this requirement.
Moreover, in the performative setting, the ambiguity set is constructed from data collected after a model is deployed, and is therefore decision-dependent. 
In this paper, we consider a more general ambiguity set by allowing its size, or equivalently, the Lagrangian penalty, to depend on the decision model. 
Under this generalized ambiguity set, we establish convergence guarantees for both repeated risk minimization and repeated gradient descent, while explicitly accounting for approximation errors in the inner maximization problem.
A detailed comparison of our approach between existing baselines is provided in Table~\ref{tab:comparison}.

The remainder of this article is structured as follows: Section~\ref{sec:problem} outlines the preliminaries and problem formulation. Section~\ref{sec:alg} develops the DR-RRM and DR-RGD algorithms for the performative prediction and analyzes their convergence. Section~\ref{sec:analysis} analyzes the suboptimality gap between the performative stable point and the performative optimum. Section~\ref{sec:simulation} demonstrates the efficacy of the algorithms through numerical simulations. Section~\ref{sec:conclusion} concludes this work. Some of the proofs are given in the Appendix.

\section{Preliminaries and problem statement}
\label{sec:problem}

\subsection{Performative prediction}
Let $\theta \in \Theta$ denote the model parameter and $\Theta$ the parameter space. Let the random samples $\xi \in \Xi$ be drawn from a distribution $\mathbb{P}$, where  $\Xi \subset \mathbb{R}^d$ is a compact and convex set.
Denote the diameter of the admissible set $\Xi $ as $D_\xi = \sup_{x,y \in \Xi } \| x-y\|$. 
Performativity means that the deployment of a predictive model will affect the observed distribution, denoted by $\mathbb{P}(\theta)$.  
Given a loss function $l(\theta,\xi): \Theta\times \Xi \rightarrow \mathbb{R}$, the prediction performance is evaluated by performative risk $ \mathbb{E}_{\xi \sim \mathbb{P}(\theta)}[l(\theta,\xi)]$,  
which represents the expected loss over the distribution $\mathbb{P}(\theta)$ it induces.

\subsection{Distributionally robust optimization}
In real-world scenarios, the distribution of interest $\mathbb{P}(\theta)$ is often unobservable. We typically only have access to datasets that contain finite samples independently drawn from  $\mathbb{P}(\theta)$, which yields an empirical estimate $\hat{\mathbb{P}}(\theta)$. To mitigate the risk of distribution mismatch, we employ a DRO approach that minimizes the worst-case risk over an ambiguity set centered at the empirical distribution. We use optimal transport cost to quantify the discrepancy between a candidate distribution $\mathbb{P}$
and the empirical distribution $\hat{\mathbb{P}}(\theta)$: 
\begin{equation*}
\mathcal{W}_c\big(\mathbb{P},\hat{\mathbb{P}}(\theta)\big)=\inf_{\mathbb{D}\in \Gamma(\mathbb{P},\hat{\mathbb{P}}(\theta))}  \mathbb{E}_{(\xi,\zeta) \sim \mathbb{D}}[c(\xi,\zeta)] ,
\end{equation*}
where $\Gamma(\mathbb{P},\hat{\mathbb{P}}(\theta))$ denotes the set of joint distributions of random variables $(\xi,\zeta)$ with $\xi\sim \mathbb{P}$ and $\zeta \sim \hat{\mathbb{P}}(\theta)$. The transportation cost function $c: \Xi\times \Xi \rightarrow [0, \infty)$ is nonnegative, lower semi-continuous and satisfies $c(z,z)=0$, for $z \in \Xi$.
Denote $\mathcal{D}(\Xi)$ as the set of all distributions supported on $\Xi$.
The ambiguity set, defined as 
\begin{eqnarray}\label{eq:amb set}
 \mathcal{B}_{\rho(\theta)}(\hat{\mathbb{P}}(\theta)):=\{\mathbb{P}\in \mathcal{D}(\Xi): \mathcal{W}_c(\mathbb{P},\hat{\mathbb{P}}(\theta))\le \rho(\theta) \},   
\end{eqnarray}
contains all distributions within $\rho(\theta)$-distance from $\hat{\mathbb{P}}(\theta)$, where $\rho(\theta) \ge  0$ is a user-specified radius. In DRO, a larger $\rho$ represents higher robustness and a more conservative model. 

\begin{figure}[t!]
    \centering
    \resizebox{0.8\linewidth}{!}{%
    \begin{tikzpicture}[
        font=\footnotesize, 
        >=Stealth, 
        node distance=1.2cm, 
        box/.style={
            draw, 
            thick, 
            rectangle, 
            rounded corners=2pt, 
            align=center, 
            inner sep=6pt, 
            fill=white
        },
        line/.style={
            draw, 
            thick, 
            ->,
            rounded corners=5pt
        }
    ]


        \node[box] (loss) {
            DR performative risk \\ 
            $\sup_{ \mathbb{P}\in \mathcal{B}_{\rho(\theta)}(\hat{\mathbb{P}}(\theta))}  \mathbb{E}_{\xi \sim \mathbb{P}}[l(\theta,\xi)]$
        };

        \node[box, right=of loss] (decision) {Decision $\theta$};


        \node[box, below=1.0cm of $(loss.south)!0.5!(decision.south)$] (dist) {Distribution $\hat{\mathbb{P}}(\theta)$};

        \node[draw, dashed, thick, inner sep=15pt, fit=(loss) (decision), green!30!black, rounded corners=4pt] (container) {};

        \node[anchor=north east, green!30!black, font=\bfseries\footnotesize, inner sep=2pt, xshift=-2pt, yshift=-2pt] 
            at (container.north east) {Decision-maker};

        
        \draw[line] (loss.east) -- (decision.west);

        \draw[line] (decision.south) |- (dist.east);

        \draw[line] (dist.west) -| (loss.south);

    \end{tikzpicture}%
    }
    \caption{Distributionally robust (DR) performative prediction diagram. }
    \label{fig:dro_labeled}
\end{figure}

\subsection{Distributionally robust performative prediction}
Given the ambiguity set \eqref{eq:amb set}, the distributionally robust performative prediction is formulated as 
\begin{equation}\label{eq:primal DRO}
\mathop{\min}_{\theta\in \Theta} \sup_{ \mathbb{P}\in \mathcal{B}_{\rho(\theta)}(\hat{\mathbb{P}}(\theta))}  \mathbb{E}_{\xi \sim \mathbb{P}} [l(\theta,\xi)].
\end{equation}
Fig.~\ref{fig:dro_labeled} illustrates the distributionally robust performative prediction diagram.  
Deploying a decision $\theta$ induces an  empirical distribution $\hat{\mathbb{P}}(\theta)$. The decision-maker then computes a new model parameter $\theta'$ by minimizing the worst-case risk over a Wasserstein ambiguity set centered at $\hat{\mathbb{P}}(\theta)$. This updated decision $\theta'$ subsequently induces a new data distribution $\hat{\mathbb{P}}(\theta')$, thereby giving rise to a feedback loop between model updates and distributional shifts. 

Inspired by \cite{perdomo2020performative,jia2025distributionally}, we define two solution concepts for distributionally robust performative prediction: the distributionally robust performative optimum and stable point.
A decision $\theta_{\rm po}$ is distributionally robust performative optimal if it minimizes the robust performative risk: 
\begin{equation}\label{eq:performative opt}
\theta_{\rm po} \in   \mathop{\arg\min}_{\theta\in \Theta}\sup_{ \mathbb{P}\in \mathcal{B}_{\rho(\theta)}(\hat{\mathbb{P}}(\theta))}  \mathbb{E}_{\xi \sim \mathbb{P}} [l(\theta,\xi)].
\end{equation} 
In \eqref{eq:performative opt}, the parameter $\theta$ affects the objective in two ways: directly through the loss function $l(\theta,\xi)$, and indirectly through the ambiguity set $\mathcal{B}_{\rho(\theta)}(\hat{\mathbb{P}}(\theta))$ induced after its deployment. Therefore, computing $\theta_{\rm po}$ requires anticipating how the data distribution would respond to each candidate decision. Since $\hat{\mathbb{P}}(\theta)$ is typically observed only after $\theta$ has been deployed, solving \eqref{eq:performative opt} is generally difficult in practice.

As an alternative, we consider performative stability, which requires the decision to be optimal with respect to the distribution it induces. 
Formally, a decision $\theta_{\rm ps}$ is called a distributionally robust performative stable point if it satisfies:
\begin{equation*}
\theta_{\rm ps} \in \mathop{\arg\min}_{\theta\in \Theta} \sup_{ \mathbb{P}\in \mathcal{B}_{\rho(\theta_{\rm ps})}(\hat{\mathbb{P}}(\theta_{\rm ps}))}  \mathbb{E}_{\xi \sim \mathbb{P}} [l(\theta,\xi)].
\end{equation*}
In this definition, the ambiguity set is fixed by the model parameter $\theta{\rm ps}$, and the optimal solution of the resulting DRO problem is exactly $\theta{\rm ps}$. In other words, when the distribution changes in response to the deployed model, a stable point is a fixed point of the retraining procedure.

In general, the performative stable point and performative optimum do not necessarily coincide \cite{perdomo2020performative}.
Since finding the performative optimum is generally difficult, following the framework of \cite{perdomo2020performative}, this paper focuses on finding performative stable points. 

\section{Main results}\label{sec:alg}
In this section, we first reformulate the primal DRO problem \eqref{eq:primal DRO} as a minmax optimization problem via Lagrangian relaxation. Based on this formulation, we develop two algorithms to compute performative stable points, and analyze their convergence under standard regularity assumptions.

To ensure computational tractability, we consider a Lagrangian relaxation of \eqref{eq:primal DRO} \cite{sinha2017certifying}, with a penalty parameter $\lambda(\theta)$: 
\begin{equation}\label{eq:Lagrangian}
 \sup_{\mathbb{P}} \mathop\mathbb{E}_{\xi \sim \mathbb{P}} [l(\theta,\xi)] - \lambda(\theta) \mathcal{W}_c(\mathbb{P},\hat{\mathbb{P}}(\theta)) ,
\end{equation}
where $\lambda(\theta): \Theta \rightarrow \mathbb{R}$ varies with the decision $\theta$. 
Generally, $\lambda(\theta)$ is inversely related to $\rho(\theta)$, such that a larger $\lambda(\theta)$ indicates a smaller ambiguity set and less robustness.

We provide the following lemma to further reformulate \eqref{eq:Lagrangian} into a minmax optimization problem based on a robust surrogate function.
\begin{lemma}\cite{sinha2017certifying}\label{lemma:DRO}
Let functions $l(\theta,\xi):\Theta \times \Xi \rightarrow \mathbb{R}$ and $c(\xi,\zeta):\Xi \times \Xi \rightarrow \mathbb{R} $ be continuous. For any distribution $\hat{\mathbb{P}}(\theta)$ and $\lambda(\theta)\ge 0$, we have that  
\begin{align*} 
    \sup_{\mathbb{P}  } \Big\{ \mathop\mathbb{E}_{\xi \sim \mathbb{P}} [l(\theta,\xi)] - \lambda(\theta) \mathcal{W}_c(\mathbb{P},\hat{\mathbb{P}}(\theta))\Big\}  =  
\mathop\mathbb{E}_{\xi \sim \hat{\mathbb{P}}(\theta)} [f(\theta,\xi)]
\end{align*}
with $f(\theta,\xi) = \sup_{\zeta \in \Xi}  \big[l(\theta,\zeta) - \lambda(\theta) c(\xi,\zeta)\big]$.
\end{lemma}
Lemma~\ref{lemma:DRO}  transforms the distributionally robust objective into a tractable minmax problem over adversarial points. This setup enables us to derive theoretical guarantees under assumptions on the loss $l(\cdot)$ and cost $c(\cdot)$.
By Lemma~\ref{lemma:DRO}, the optimum and the stable point of the reformulated distributionally robust performative prediction \eqref{eq:Lagrangian}, respectively, satisfy
\begin{subequations}\label{eq: p points}
 \begin{align}
\theta_o &\in \mathop{\rm{argmin}}_{\theta}  \mathop\mathbb{E}_{\xi \sim \hat{\mathbb{P}}(\theta)} [f(\theta,\xi)], \label{eq:p optimum}  \\ 
  \theta_s &\in \mathop{\rm{argmin}}_{\theta}  \mathop\mathbb{E}_{\xi \sim \hat{\mathbb{P}}(\theta_s)}  [f(\theta,\xi)] . \label{eq:p stable point}
\end{align}   
\end{subequations}
To evaluate the distribution shift induced by the model deployment, a commonly used metric is the Kantorovich-Rubinstein dual formulation of the 1-Wasserstein distance, as detailed in the following lemma.
\begin{lemma}\cite{edwards2011kantorovich}\label{lemma:wasserstein}
Let $\mathbb{P}_x$ and $\mathbb{P}_y$ be two probability distributions on probability space $\Omega$. The dual form of the Wasserstein distance
is given as
\begin{equation*}
    W_1(\mathbb{P}_x, \mathbb{P}_y)=\frac{1}{K} \sup _{\|f\|_L \leq K} \{ \mathbb{E}_{x \sim \mathbb{P}_x}[f(x)]-\mathbb{E}_{y \sim \mathbb{P}_y}[f(y)] \},
\end{equation*}
for any fixed $K >0$, where $\|\cdot\|_L$ is the Lipschitz norm.
\end{lemma}

Following \cite{perdomo2020performative}, we make the following assumptions. The first assumption restricts the sensitivity of the distribution map $\hat{\mathbb{P}}(\theta)$ using $W_1$ distance metric.
\begin{assumption}\label{assumption:distribution sensitivity}
The distribution $\hat{\mathbb{P}}(\cdot)$ is $\varepsilon$-sensitive, i.e., it satisfies
\begin{equation*} 
W_1(\hat{\mathbb{P}}(\theta),\hat{\mathbb{P}}(\theta')) \le \varepsilon\|\theta-\theta'\|_2    ,
\end{equation*}
for all $\theta$, $\theta' \in \Theta$, where $\|\cdot\|_2$ denotes the Euclidean norm. 
\end{assumption}
The following two assumptions impose strong convexity and smoothness conditions on the loss function.
\begin{assumption}\label{assumption:strongly convex}
The function $l(\theta,\zeta)$ is $m$-strongly convex in $\theta$, for every $\zeta \in \Xi$, where $m$ is a positive scalar.
\end{assumption}

\begin{assumption}\label{assumption:loss function}
The loss function $l(\theta, \zeta)$ is continuously differentiable and satisfies the following Lipschitzian smoothness conditions for all $\theta, \theta' \in \Theta$ and $\zeta, \zeta' \in \Xi$:
\begin{align*}
&\|\nabla_\theta l(\theta,\zeta)-\nabla_\theta l(\theta',\zeta)\| \le L_{\theta\theta}\|\theta-\theta'\|, \nonumber \\ 
&\|\nabla_\theta l(\theta,\zeta)-\nabla_\theta l(\theta,\zeta')\| \le L_{\theta\zeta}\|\zeta-\zeta'\|, \nonumber \\ 
&\|\nabla_\zeta l(\theta,\zeta)-\nabla_\zeta l(\theta',\zeta)\| \le L_{\zeta\theta}\|\theta-\theta'\|,
\end{align*} 
where $\nabla_\theta l $ and $\nabla_\zeta l $ denote the gradients of $l(\theta,\zeta)$ with respect to $\theta$ and $\zeta$, respectively.
\end{assumption}
Assumptions~\ref{assumption:distribution sensitivity}--\ref{assumption:loss function} are standard conditions to ensure the well-posedness and stability of the retraining dynamics \cite{perdomo2020performative}.
Following standard conventions in the literature, we adopt the squared Euclidean distance 
\begin{equation}\label{eq:trans cost}
    c(\xi,\zeta) = \|\xi-\zeta\|_2^2
\end{equation}
as the transportation cost. Under \eqref{eq:trans cost}, the generalized discrepancy $\mathcal{W}_c$ corresponds to the squared 2-Wasserstein metric. Unless explicitly stated otherwise, all unsubscripted norms denote the Euclidean norm.
To simplify notation, we denote 
\begin{equation}\label{eq:robust surrogate}
\phi(\theta,\xi,\zeta) =  l(\theta,\zeta) - \lambda(\theta) c(\xi,\zeta) .   
\end{equation} Therefore, it holds that $f(\theta,\xi) = \sup_{\zeta \in \Xi} \phi(\theta,\xi,\zeta)$.
Similar to Assumptions~\ref{assumption:distribution sensitivity}--\ref{assumption:loss function}, we require the function $\lambda(\theta)$ to be bounded and smooth, as formalized in the following assumption.
\begin{assumption}\label{assumption:lambda}
Denote $\nabla_{\theta}\lambda(\theta)$ as the gradient of the function $\lambda(\theta): \Theta \rightarrow \mathbb{R}$. We have that $0< \lambda_{\min} \le \lambda(\theta) \le \lambda_{\max}$, $\|\nabla_\theta 
\lambda(\theta)\| \le L_{\lambda}$, and $\|\nabla_\theta 
\lambda(\theta)- \nabla_\theta 
\lambda(\theta')\| \le H_{\lambda}\|\theta-\theta'\|$, for some positive constants $\lambda_{\min}$, $\lambda_{\max}$, $L_\lambda$ and $H_\lambda$. 
Furthermore, define $\gamma := m - H_\lambda D_\xi^2$, we have $\gamma>0$. This means that $\phi(\theta,\xi,\zeta)$ is $\gamma$-strongly convex in $\theta$. 

\end{assumption}

To ensure computational tractability, inspired by \cite{sinha2017certifying}, we make the following concave assumption: 
\begin{assumption}\label{assumption:concave}
For every fixed $\theta \in \Theta$ and $\xi\in \Xi$, the function $\phi(\theta,\xi, \zeta)$ is $\mu$-strongly concave in $\zeta \in \Xi$. 
\end{assumption}
\noindent Under Assumption~\ref{assumption:lambda} and the transportation cost \eqref{eq:trans cost}, it is sufficient for Assumption~\ref{assumption:concave} that either $l(\theta,\zeta)$ is concave in $\zeta$, or $\lambda(\theta)$ is sufficiently large so that the function $c(\xi,\zeta)$ dominates the concavity of $l(\theta,\zeta)$ in $\zeta$.

\subsection{Performative risk minimization}
This section develops the repeated risk minimization algorithm for distributionally robust performative prediction and analyzes its convergence to the performative stable point.

\begin{algorithm}[t] 
\caption{Distributionally robust repeated risk minimization} \label{alg:risk minimization}
\begin{algorithmic}[1]
    \REQUIRE Sampling distribution $\hat{\mathbb{P}}(\theta_0)$, constraint sets  $\Theta$ and $\Xi$
    \FOR{$ {\rm{time}} \;t = 0,\dots, T$} 
    \STATE  Sample $\xi_t\sim \hat{\mathbb{P}}(\theta_t)$
    \STATE  Find $\epsilon$-approximate maximizers $z(\theta_t,\xi_t)$, as in \eqref{eq:maximizer z} 
    \STATE  Update decision $\theta_{t+1} $ via \eqref{eq:Alg risk minimization}
    \ENDFOR
\end{algorithmic}
\end{algorithm}
 
At iteration $t$, we observe $\xi_t$, which is sampled from the true distribution $\mathbb{P}(\theta_{t})$. This realization constitutes a empirical distribution estimate $\hat{\mathbb{P}}(\theta_{t})$. 
Then, the algorithm updates the model by minimizing the worst-case risk over the ambiguity set centered at the distribution $\hat{\mathbb{P}}(\theta_{t})$. 
By Lemma~\ref{lemma:DRO}, we first compute a maximizer for the inner objective function: 
\begin{align}\label{eq:maximizer z}
\zeta^\ast(\theta_t,\xi_t) = \mathop{\arg\max}_{\zeta\in \Xi} \phi(\theta_t,\xi_t,\zeta) ,
\end{align}
where $\zeta^\ast(\theta_t,\xi_t)$ represents a realization of the worst-case distribution.  
By the strongly concavity of $\phi(\theta,\xi,\cdot)$, $\zeta^\ast(\theta_t,\xi_t)$ is unique for every pair $(\theta_t,\xi_t)$.
Ideally, when the function $\phi$ admits a closed-form solution, an exact maximizer $ \zeta^{\ast}(\theta_t,\xi_t)$ can be derived.  
However, when its explicit form is unknown,  the maximizer is obtained numerically. Denote $\mathcal{Z}_\epsilon(\theta, \xi) = \{ \zeta \mid \|\zeta - \zeta^*(\theta, \xi)\| \le \epsilon \}$ as the local uncertainty set around $\zeta^*(\theta, \xi)$,  where $\epsilon>0$ is the tolerance. Then, the maximizer obtained numerically is denoted as $z(\theta_t,\xi_t) \in \mathcal{Z}_\epsilon(\theta_t, \xi_t)$.
We retain the $\epsilon$-approximation formulation to ensure the generality of the framework.
Finally, we update the model parameter $\theta_t$ by minimizing distributionally robust performative risk: 
\begin{equation}\label{eq:Alg risk minimization}
\theta_{t+1}   = \mathop{{\arg\min}}_{\theta\in \Theta} \mathop{\mathbb{E}}_{\xi_t\sim\hat{\mathbb{P}}(\theta_t)}\big[ \phi(\theta,\xi_t,z(\theta_t,\xi_t))\big],
\end{equation} 
with initial value $\theta_0$.
To summarize, as outlined in Algorithm \ref{alg:risk minimization}, the procedure iteratively identifies adversarial points $z(\theta_t,\xi_t)$ and updates the model parameters $\theta_t$ by minimizing risk evaluated on these adversarial points.

Consider the scenario in which \eqref{eq:Alg risk minimization} is solved without approximation error, this minimization problem is still not identical to \eqref{eq:p stable point}, as the random variable $\xi_t$ in adversarial points $\zeta^\ast(\theta_t,\xi_t)$ follows the distribution $\hat{\mathbb{P}}(\theta_t)$.
However, when the iteration sequence generated by \eqref{eq:Alg risk minimization} converges to a stable point $ \theta_s$, it satisfies $\theta_s = \mathop{{\arg\min}}_{\theta\in \Theta} \mathop{\mathbb{E}}_{\xi \sim\hat{\mathbb{P}}(\theta_s)}\big[ \phi(\theta,\xi ,z(\theta_s,\xi ))\big]$. 
This solution coincides with the definition of the performative stable point in \eqref{eq:p stable point}.

In the following lemmas, we analyze the Lipschitzian smoothness
conditions of functions $\phi$ and $f$, respectively. They are used for the convergence analysis of Algorithm~\ref{alg:risk minimization}.
\begin{lemma}\label{lemma:phi Lipschitzs}
Let Assumptions~\ref{assumption:loss function}--\ref{assumption:lambda} hold. 
Denote $\nabla_\theta\phi$ and $\nabla_\zeta \phi$ as the gradient of $\phi$ with respect to $\theta$ and $\zeta$, respectively. 
For all $\theta,\theta' \in \Theta$, $\xi,\xi',\zeta,\zeta' \in \Xi$, the function $\phi(\theta,\xi,\zeta)$ satisfies 
\begin{subequations}
\begin{align}
&\|\nabla_\theta\phi(\theta,\xi,\zeta)-\nabla_\theta\phi(\theta',\xi,\zeta)\| \le L_{\theta\theta}^\phi\|\theta-\theta'\| ,\label{eq:phi lip theta_theta} \\
&\|\nabla_\theta\phi(\theta,\xi,\zeta)-\nabla_\theta\phi(\theta,\xi',\zeta)\| \le L_{\theta\xi}^\phi\|\xi-\xi'\| \label{eq:phi lip theta_xi}, \\
&\|\nabla_\theta\phi(\theta,\xi,\zeta)-\nabla_\theta\phi(\theta,\xi,\zeta')\| \le L_{\theta\zeta}^\phi\|\zeta-\zeta'\| ,\label{eq:phi lip theta_zeta} \\
&\|\nabla_\zeta\phi(\theta,\xi,\zeta)-\nabla_\zeta\phi(\theta',\xi,\zeta)\| \le L_{\zeta\theta}^\phi\|\theta-\theta'\|  \label{eq:phi lip zeta_theta}, \\ 
&\|\nabla_\zeta\phi(\theta,\xi,\zeta)-\nabla_\zeta\phi(\theta,\xi',\zeta)\| \le L_{\zeta\xi}^\phi\|\xi-\xi'\|,  \label{eq:phi lip zeta_xi}
\end{align}
\end{subequations}
 with $L_{\theta\theta}^\phi= L_{\theta\theta}+H_\lambda D_\xi^2$, $L_{\theta\xi}^\phi= 2L_\lambda D_\xi$, $L_{\theta\zeta}^\phi=L_{\theta\zeta}+2L_{\lambda}D_\xi$,  $L_{\zeta\theta}^\phi=L_{\zeta\theta} + 2L_\lambda D_\xi $, and $ L_{\zeta\xi}^\phi = 2\lambda_{\max}$.  
\end{lemma}

\begin{lemma}\label{lemma:f Lipschitz}
Let Assumptions~\ref{assumption:loss function}--\ref{assumption:concave} hold.  We have that $\nabla_\theta f(\theta,\xi)  = \nabla_\theta\phi(\theta,\xi,\zeta^\ast(\theta,\xi))$. Moreover, the following claims hold:
\begin{enumerate}
\item $\|\zeta^\ast(\theta,\xi)- \zeta^\ast(\theta',\xi)\| \le L_{\zeta\theta}^\phi/\mu\|\theta-\theta'\|$, and 
$\|\nabla_\theta f(\theta,\xi)-\nabla_\theta f(\theta',\xi)\| \le L_{\theta\theta}^f\|\theta-\theta' \| $   with $L_{\theta\theta}^f= L_{\theta\theta}^\phi +L_{\theta\zeta}^\phi L_{\zeta\theta}^\phi /\mu$;
\item $\|\zeta^\ast(\theta,\xi)- \zeta^\ast(\theta,\xi')\| \le L_{\zeta\xi}^\phi/\mu\|\xi-\xi'\|$, and 
$\|\nabla_\theta f(\theta,\xi)-\nabla_\theta f(\theta,\xi')\| \le L_{\theta\xi}^f\|\xi-\xi' \| $ with $L_{\theta\xi}^f = L_{\theta\xi}^\phi + L_{\theta\zeta}^\phi L_{\zeta\xi}^\phi/\mu$.
\end{enumerate}
\end{lemma}
The proofs of Lemmas~\ref{lemma:phi Lipschitzs} and~\ref{lemma:f Lipschitz} are given in the Appendix. 
We need to distinguish the smoothness of $\nabla_{\theta}\phi$ from that of $\nabla_{\theta}f$. While the former treats $\zeta$ as fixed, the latter accounts for the implicit dependence of $\zeta^\ast(\theta,\xi)$ on $\theta$.

In the following, we show that the repeated risk minimization operation is a contraction mapping, and then analyze the convergence rate of Algorithm~\ref{alg:risk minimization}. 
\begin{theorem}\label{theo:risk minimization} Let Assumptions~\ref{assumption:distribution sensitivity}--\ref{assumption:concave} hold.  
Denote $ \kappa_{\rm rm} = \big(L_{\theta\xi}^f\varepsilon+ L_{\theta\zeta}^\phi L_{\zeta\theta}^\phi/\mu\big)/\gamma  $. 
When $\kappa_{\rm rm}<1$, we have that
\begin{enumerate}
\item The sequence $\theta_t$ generated by \eqref{eq:Alg risk minimization} converges linearly to a neighborhood of the unique stable point $\theta_s$: 
\begin{equation*} 
    \|\theta_t - \theta_s \| \le \kappa_{\rm rm}^t \|\theta_0 - \theta_s\| + \frac{ L_{\theta\zeta}^\phi \epsilon}{\gamma(1-\kappa_{\rm rm})},
\end{equation*} 
for all $t \ge 0$. 
\item When the inner maximization is exact (i.e., $\epsilon=0$), the sequence $\theta_t$ converges linearly to the fixed point $\theta_s$ with rate $\kappa_{\rm rm}$.
\end{enumerate}
\end{theorem}
The proof of Theorem~\ref{theo:risk minimization} is given in the Appendix. 
\begin{remark}
Theorem~\ref{theo:risk minimization} demonstrates that the iterates $\theta_t$ converge faster if the function is smooth, strongly convex in $\theta$, and strongly concave in $\zeta$ (smaller  $L_{\theta\theta}^\phi$, $L_{\theta\zeta}^\phi$ and $L_{\zeta\theta}^\phi$,  larger $\gamma$ and $\mu$), and the distribution map is insensitive (smaller $\varepsilon$). 
Specifically, the distribution shift induced by the update of $\theta$ must be sufficiently small such that the algorithm can consistently contract toward the stable point.
Moreover, due to the approximation error $\epsilon$, the sequence converges to a neighborhood of the stable point. 
\end{remark}

\subsection{Performative gradient descent}
Repeated risk minimization assumes access to an exact optimization oracle, which is often impractical. In this section, we relax this requirement, and develop a gradient descent approach to find the stable point. 

Similar to Algorithm~\ref{alg:risk minimization}, we update the decision based on the empirical distribution $\hat{\mathbb{P}}(\theta_t)$ resulting from the previous model $\theta_t$, but instead utilize gradient descent to optimize the performative risk.
\begin{algorithm}[t] 
\caption{Distributionally robust repeated gradient descent} \label{alg:gradient descent}
\begin{algorithmic}[1]
    \REQUIRE Sampling distribution $\hat{\mathbb{P}}(\theta_0)$, constraint sets $\Theta$ and  $\Xi$, step size $\eta$
    \FOR{$ {\rm{time}} \;t = 0,\dots, T$} 
    \STATE  Sample $\xi_t\sim \hat{\mathbb{P}}( \theta_t)$
    \STATE  Find $\epsilon$-approximate maximizers $z(\theta_t,\xi_t)$, as in \eqref{eq:maximizer z}
    \STATE Construct gradient estimate $g_t$, as in \eqref{eq:gradient estimate}  
    \STATE Update decision $\theta_{t+1}$ via \eqref{eq:repeat gradient descent}
    \ENDFOR
\end{algorithmic}
\end{algorithm}
At iteration $t$, we observe $\xi_t \sim \hat{\mathbb{P}}(\theta_t)$. Then, we find the $\epsilon$-approximate maximizer $z(\theta_t,\xi_t)$ of the function $\phi(\theta_t,\xi_t,\cdot)$ via \eqref{eq:maximizer z}. 
Given $ \phi(\theta_t,\xi_t,z(\theta_t,\xi_t))$, the gradient estimate is given as 
\begin{equation}\label{eq:gradient estimate}
g_t = \mathbb{E}_{\xi_t} \big[\nabla_\theta  \phi(\theta_t,\xi_t,z(\theta_t,\xi_t))\big].
\end{equation}
Then, the decision is updated according to 
\begin{equation}\label{eq:repeat gradient descent}
\theta_{t+1} = {\rm Proj}_{\Theta}\big(\theta_t - \eta g_t\big),
\end{equation}
with $\theta_0$ is the initial value, 
 $\eta>0$ the step size, and ${\rm Proj}_{\Theta}$ the projection operator onto $\Theta$. 
Algorithm~\ref{alg:gradient descent} outlines the repeated gradient descent method, which iteratively identifies the adversarial points  $z(\theta_t,\xi_t)$ and updates the decision via the gradient descent \eqref{eq:repeat gradient descent}. 

The following theorem analyzes the convergence of Algorithm~\ref{alg:gradient descent}. 
\begin{theorem}\label{theo:gradient descent} Let Assumptions~\ref{assumption:distribution sensitivity}--\ref{assumption:concave} hold.  Denote $\beta =L_{\theta\theta}^f+ \varepsilon L_{\theta\xi}^f $, and   
$\kappa_{\rm gd}=\sqrt{\eta^2\beta^2+2\eta(\varepsilon L_{\theta\xi}^f-\gamma)+1}$. When $\gamma>\varepsilon L_{\theta\xi}^f$, and the step size $0<\eta<2(\gamma-\varepsilon L_{\theta\xi}^f)/\beta^2$, we have $\kappa_{\rm gd}<1$, and the following claims:
\begin{enumerate}
\item The sequence $\theta_t$ generated by \eqref{eq:repeat gradient descent} converges linearly to a neighborhood of $\theta_s$ with rate  $\kappa_{\rm gd}$:
\begin{equation*}
\|\theta_t - \theta_s\| \le \kappa_{\rm gd}^t\|\theta_0-\theta_s\| + \frac{ \eta L_{\theta\zeta}^\phi \epsilon}{1-\kappa_{\rm gd}},
\end{equation*}
for all $t \ge 0$. 
\item When the inner maximization is exact, i.e., $\epsilon=0$, the sequence $\theta_t$ converges linearly to the fixed point $\theta_s$ with rate $\kappa_{\rm gd}$.
\end{enumerate}
\end{theorem}
The proof of Theorem~\ref{theo:gradient descent} is given in the Appendix.

\begin{remark}\label{remark:gradient descent}
Similar to  Theorem~\ref{theo:risk minimization}, the iterates $\theta_t$ converge faster when the objective function is smooth, strongly convex, and when the distribution map exhibits low sensitivity.  
In risk minimization, the convergence is determined by the stability of the optimal decision under distributional shifts. 
For gradient descent, maintaining a linear rate of convergence requires that the step size $\eta$ be sufficiently small relative to the regularity of the gradient.   
Additionally, Theorems~\ref{theo:risk minimization} and~\ref{theo:gradient descent} show that the limiting convergence neighborhoods are determined by the approximation error $\epsilon$ and the corresponding contraction margins  $\gamma_{\rm rm}$ and $\gamma_{\rm gd}$. Specifically, the neighborhood in Theorem~\ref{theo:risk minimization} is governed by  $\gamma_{\rm rm}$, whereas the neighborhood in Theorem~\ref{theo:gradient descent} is governed jointly by $\gamma_{\rm gd}$ and the algorithmic step size $\eta$.
\end{remark}



\section{Suboptimality guarantee}\label{sec:analysis}
The preceding section developed repeated risk minimization and gradient descent algorithms, and showed that both algorithms successfully converge to a unique robust performative stable point. Building on these results, this section evaluates the stable point by characterizing its distance to the optimal point and the corresponding excess performative risk.
To formally relate the stable and optimal points, we introduce the decoupled performative risk \cite{perdomo2020performative}: 
\begin{equation*} 
{\rm DPR}(\theta,\theta'):= \mathbb{E}_{\xi \sim \mathbb{\hat{P}(\theta)}}[f(\theta',\xi)].
\end{equation*} 
Accordingly, the performative stable point and the performative optimum can be expressed as 
\begin{align*} 
\theta_s &= {\arg\min}_{\theta\in \Theta} {\rm DPR}(\theta_s,\theta), \nonumber \\ 
\theta_o &= {\arg\min}_{\theta\in \Theta} {\rm DPR}(\theta,\theta) .
\end{align*}

We provide the following lemma to analyze the Lipschitzness of the function $f$ with respect to $\xi$. 
\begin{lemma}\label{lemma:L_xi_f Lipschitz}
Let Assumptions~\ref{assumption:loss function}--\ref{assumption:concave} hold.
Assume that the loss function $l(\theta,\zeta): \Theta \times\Xi \rightarrow \mathbb{R}$ is $L_\zeta$-Lipschitz in $\zeta$. Then, the function $f(\theta,\xi): \Theta \times \Xi \rightarrow \mathbb{R}$ is $L_\xi^f$-Lipschitz in $\xi$ with $L_{\xi}^f =   2\lambda_{\max}(L_\zeta+2\lambda_{\max}D_\xi)/\mu + 2\lambda_{\max}D_\xi$.
\end{lemma}
\noindent The proof of Lemma~\ref{lemma:L_xi_f Lipschitz} is given in the Appendix.

The following theorem analyzes the suboptimality gap between the theoretical performative stable point $\theta_s$ and the performative optimum $\theta_o$.  
\begin{theorem}\label{theorem:suboptimality}
Let Assumptions~\ref{assumption:distribution sensitivity}--\ref{assumption:concave} hold. Assume that the loss function $l(\theta,\zeta)$ is $L_\zeta$-Lipschitz in $\zeta$.  Then, the performative stable point $\theta_s$ and the performative optimum $\theta_o$ satisfy
\begin{align}
\|\theta_s-\theta_o\| &\le 2\varepsilon L_\xi^f/\gamma, 
\label{eq:suboptimality gap} \\ 
{\rm DPR}(\theta_s,\theta_s)-{\rm DPR}(\theta_o,\theta_o) &\le  2(\varepsilon L_\xi^f)^2/\gamma. \nonumber
\end{align}
\end{theorem}
\noindent\textit{Proof}. 
Since the function ${\rm DPR}(\theta_s,\cdot)$ is $\gamma$-strongly convex, we have that 
\begin{align*} 
\frac{\gamma}{2}\|\theta_o-\theta_s\|^2 \le &  {\rm DPR}(\theta_s,\theta_o) - {\rm DPR}(\theta_s,\theta_s) \nonumber \\ 
\le & {\rm DPR}(\theta_s,\theta_o) - {\rm DPR}(\theta_o,\theta_o) \nonumber \\ 
 \le & \varepsilon L_{\xi}^f \|\theta_s-\theta_o\|,
\end{align*}
where the second inequality follows from ${\rm DPR}(\theta_s,\theta_s) \ge {\rm DPR}(\theta_o,\theta_o)$. The last inequality follows from the $\varepsilon$-sensitivity of the distribution map $\hat{\mathbb{P}}(\cdot)$ and the function $f(\theta,\xi)=\phi(\theta,\xi,\zeta^\ast(\theta,\xi))$ is $L_{\xi}^f$-Lipschitz in $\xi$, as in Lemma~\ref{lemma:L_xi_f Lipschitz}. 
Thus, we obtain \eqref{eq:suboptimality gap}. 

Furthermore, the suboptimality bound of the decision $\theta_s$ is given as 
\begin{align}\label{eq:function subopt bound} 
& {\rm DPR}(\theta_s,\theta_s) - {\rm DPR}(\theta_o,\theta_o)  \le  {\rm DPR}(\theta_s,\theta_o) - {\rm DPR}(\theta_o,\theta_o) \nonumber\\
& \le  \varepsilon L_\xi^f\|\theta_s - \theta_o\| \le \frac{2(\varepsilon L_\xi^f)^2}{\gamma}, 
\end{align}
where the first inequality follows from that $\theta_s$ minimizes ${\rm DPR}(\theta_s,\cdot)$. The second inequality follows from the $\varepsilon$-sensitivity of the distribution map $\hat{\mathbb{P}}(\cdot)$ and the function $f(\theta,\xi)$ is $L_{\xi}^f$-Lipschitz in $\xi$. The last inequality follows from substituting \eqref{eq:suboptimality gap} into \eqref{eq:function subopt bound}. 
\hfill \qed 

\begin{remark}
Theorem~\ref{theorem:suboptimality} analyzes the equilibrium properties through $\|\theta_o - \theta_s\|$, which is an irreducible approximation gap inherent to decision-dependent problems.
It arises from the myopia of iterative methods that optimize based on current observations without anticipating future distributional shifts. This gap is an intrinsic property that is independent of the chosen algorithm. 
\end{remark}

\section{Simulations}\label{sec:simulation}
This section demonstrates the performance of the proposed distributionally robust framework through a strategic classification task. In dynamic credit scoring environments, a financial institution assesses borrowers' creditworthiness before making loan decisions   \cite{perdomo2020performative}. Conversely, this assessment mechanism incentivizes agents to strategically adapt their features to elicit a favorable credit decision.
We adopt the  ``Give Me Some Credit" Kaggle dataset \cite{GiveMeSomeCredit}
as the base distribution $\mathcal{D}$. 
This dataset consists of feature vectors $x \in \mathbb{R}^{m}$, which represent applicants' characteristics, such as monthly income, age, and the number of existing loans. The outcome variable $y \in \{0,1\}$ indicates default status, where $y=1$ denotes that the individual defaulted on the loan, and $y=0$ otherwise.
Let $\xi = [x^\top, y]^\top$ denote a data point. 

Applicants' feature set comprises both strategic features  $S \subseteq \{1, \dots, m\}$, which they can manipulate (e.g., utilization of credit lines, number of open credit lines, and number of real estate loans), and non-strategic features (e.g., age).
In our setting, $|S|=3$. In response to the classifier $\theta_S$, applicants adjust their strategic features according to
$$x_{S}' = x_S - \varepsilon\theta_S,$$
where $x_S, \theta_S \in \mathbb{R}^{|S|}$, and the parameter $\varepsilon$ regulates the sensitivity of the feature distribution. This reaction generates a new sample distribution $\mathcal{D}(\theta_{S})$.

Denote the sample size as $n$.
Denote the sample as $\xi_i = [x_i^\top, y_i]^\top$, for $i = \{1,\dots,n\}$, which are drawn from the current distribution $\mathcal{D}(\theta_S)$.
At each iteration, the institution trains the classifier model by minimizing the regularized empirical risk \cite{perdomo2020performative}:
\begin{equation*}
\mathcal{L}(\theta,\xi_i) = \frac{1}{n}\sum_{i=1}^n \left[ l(\theta,\xi_i) \right] + \frac{\gamma_{\rm reg}}{2}\|\theta\|^2,
\end{equation*}
where $l(\theta, \xi_i) = -y_i\theta^\top x_i + \log(1+\exp(\theta^\top x_i))$ is a logistic regression model used to predict default probability. Moreover,  $\gamma_{\rm reg} > 0$ is a regularization coefficient. 
The smoothness of $\mathcal{L}(\cdot,\cdot)$ is verified in \cite{perdomo2020performative}. 

To account for the discrepancy between the empirical distribution and the true population, we use a distributionally robust framework. Here, the learner minimizes a robust surrogate:
\begin{equation*}
\min_{\theta\in \Theta} \frac{1}{n} \sum_{i=1}^n \sup_{\zeta_i\in\Xi  } \left( \mathcal{L}(\theta,\zeta_i) - \lambda(\theta)c(\xi_i,\zeta_i) \right),
\end{equation*}
where $c(\cdot,\cdot)$ is the transport cost, as defined in \eqref{eq:trans cost}. The parameter  $\lambda(\theta) = 30 + 0.1\|\theta\|^2$ regulates the radius of the ambiguity set, and the set $\Xi$ represents the data domain. 

\begin{figure}
    \centering
    \includegraphics[width=0.96\linewidth]{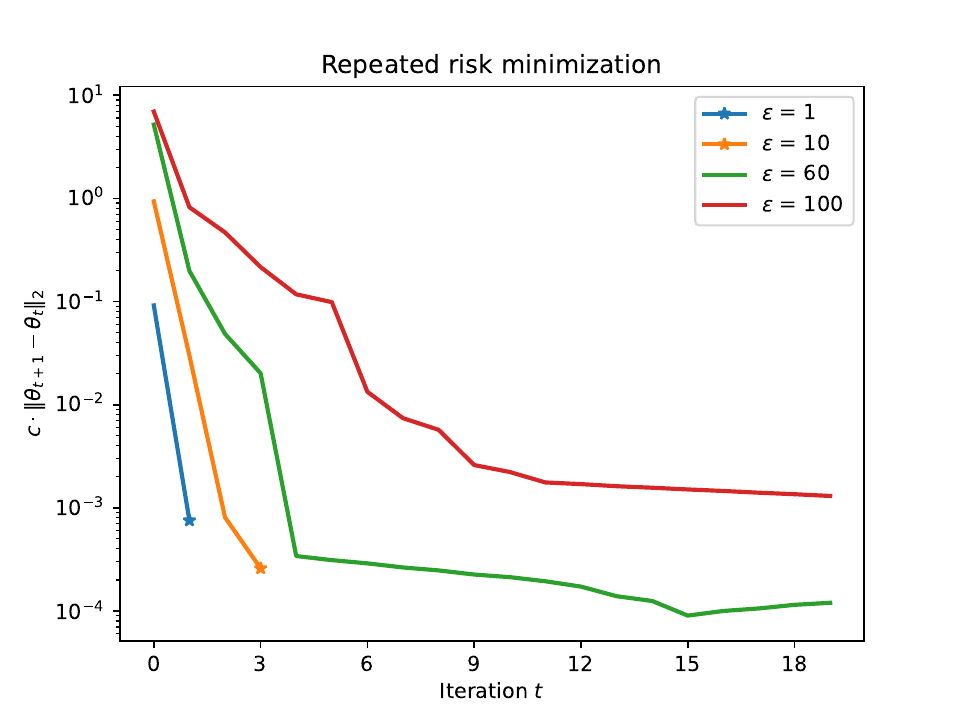}
    \caption{Convergence of the distributionally robust performative risk minimization method (Algorithm~\ref{alg:risk minimization}) for varying $\varepsilon$-sensitivity parameters. We add a marker if the distance between iterates is numerically zero at the next iteration.} 
    \label{fig:theta gap rm}
\end{figure}
\begin{figure}
    \centering
    \includegraphics[width=0.96\linewidth]{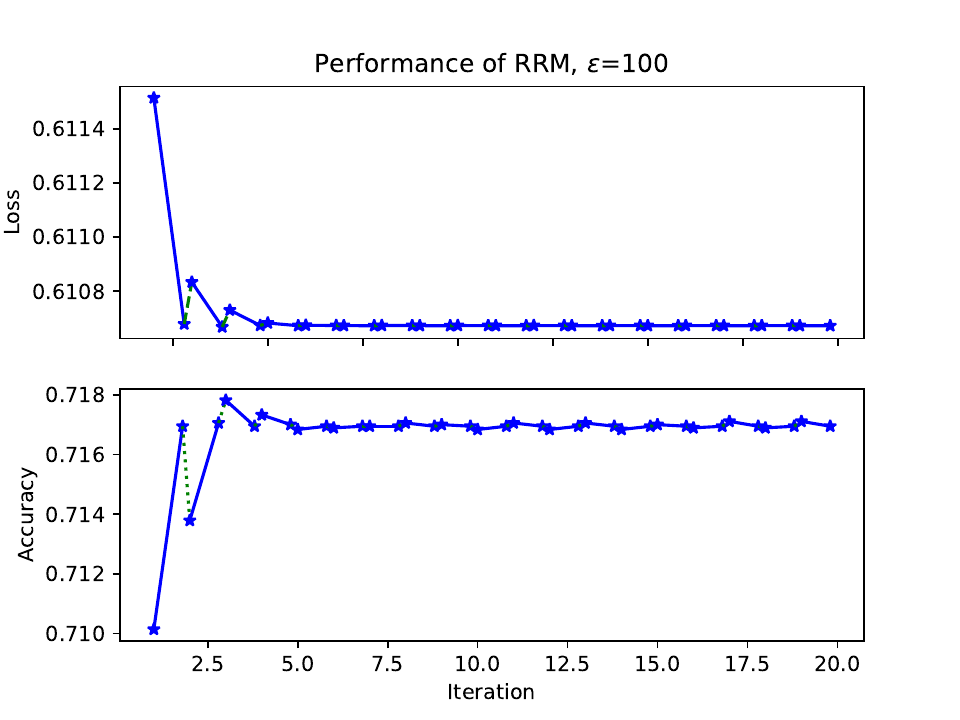}
    \caption{Performance evolution of Algorithm~\ref{alg:risk minimization} under strategic sensitivity $\varepsilon = 100$. Solid blue lines indicate the optimization phase, and dotted green lines indicate the distribution shift after classifier deployment.}
    \label{fig:acc and loss eps=60}
\end{figure}
\begin{figure}
    \centering
    \includegraphics[width=0.96\linewidth]{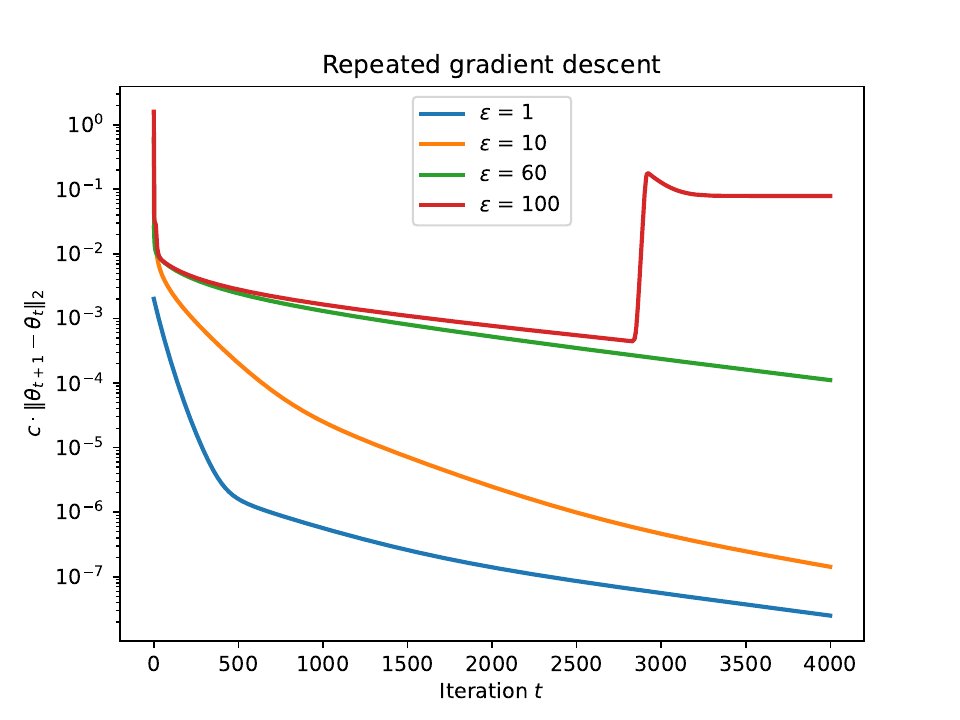}
    \caption{Convergence of the distributionally robust performative gradient descent method (Algorithm~\ref{alg:gradient descent}) for varying $\varepsilon$-sensitivity parameters.} 
    \label{fig:theta gap gd}
\end{figure}
We first examine the convergence rate of the distributionally robust risk minimization scheme (Algorithm~\ref{alg:risk minimization}) across various sensitivity levels $\varepsilon \in \{1, 10, 60, 100\}$. The error tolerance for the internal maximizer is set to $\epsilon = 10^{-9}$. In Fig.~\ref{fig:theta gap rm}, we normalize the parameter distance using the scaling factor $c = \|\theta_{0,S}\|_2^{-1}$ with $\theta_{0,S}$ being the initial value of $\theta_{S}$. 
Fig.~\ref{fig:theta gap rm} shows that  Algorithm~\ref{alg:risk minimization} achieves convergence within $20$ iterations. Notably, for lower sensitivity levels ($\varepsilon \in \{1, 10\}$), the trajectories terminate abruptly at iterations 2 and 3, respectively, indicating that the algorithm reached exact convergence. Moreover, larger $\varepsilon$ values induce more pronounced oscillations.
Furthermore, Fig.~\ref{fig:acc and loss eps=60} shows the performance of Algorithm~\ref{alg:risk minimization} under strategic sensitivity $\varepsilon = 100$. It depicts the evolution of the empirical loss and classification accuracy over 20 iterations. Despite the distributional shifts induced by strategic feature manipulation, the process stabilizes within $10$ iterations. 

Fig.~\ref{fig:theta gap gd} illustrates the convergence behavior of the distributionally robust gradient descent method (Algorithm~\ref{alg:gradient descent}) under identical experimental conditions. Compared to the risk-minimization approach, gradient descent converges more slowly. For all experiments, we use a constant step size of $\eta = 10^{-2}$. Notably, the convergence speed degrades as the strategic sensitivity $\varepsilon$ increases. When $\varepsilon = 100$, the parameter gap $\|\theta_{t+1} - \theta_t\|_2$ jumps to $0.1$ during the iterations, which means that the iterates $\theta_t$ begin to oscillate within a specific region. This behavior empirically validates Remark~\ref{remark:gradient descent}, which states that the stability of gradient descent depends on the alignment between the step size and the gradient's regularity.

To demonstrate the robustness of the Wasserstein DRO (W-DRO) framework, Fig.~\ref{fig:compare_acc_algs} compares the pre- and post-deployment classification accuracies of Algorithm~\ref{alg:risk minimization} against those of empirical risk minimization (ERM) and DRO based on KL-divergence (KL-DRO) \cite{xue2024distributionally}, across varying sensitivity levels ($\varepsilon \in [0, 80]$). 
We partition the data set into training and test sets. 
The classifiers are trained on a perturbed version of the training data and subsequently evaluated on the test set.
``pre'' and ``post'' denote before and after the data distribution adapts to the
model deployments, respectively. 
In the pre-deployment phase, the accuracy decreases as the sensitivity parameter $\varepsilon$ increases. This degradation occurs because the robust classifiers trade off baseline accuracy to establish a safety margin against the anticipated distribution shifts.  
In the post-deployment phase, ERM outperforms DRO when $\varepsilon = 0$. However, for $\varepsilon > 0$, both W-DRO and KL-DRO consistently outperform ERM, as their decisions $\theta$ are more resilient to different data sets. Furthermore, as $\varepsilon$ continues to increase, the post-deployment accuracy of W-DRO improves. This arises because the robust classifier reshapes the data distribution and forces strategic user behavior into a more linearly separable space, thereby making the shifted data easier to classify at equilibrium.

In Fig.~\ref{fig:compare radius}, we further investigate the sensitivity of the proposed W-DRO framework to the regularization parameter $\lambda(\theta) = \lambda_c + 10^{-7} \cdot \|\theta\|^2$. In this paper, a larger $\lambda_c$ corresponds to a smaller ambiguity set, which drives the W-DRO solution to converge toward the ERM baseline. 
The results demonstrate that W-DRO  consistently outperforms ERM across all tested regularization levels. Furthermore, as $\lambda_c$ increases, we observe a marginal degradation in the accuracy. This behavior confirms that the distributional robustness framework effectively counters strategic masking.

\begin{figure}
    \centering
    \includegraphics[width=0.96\linewidth]{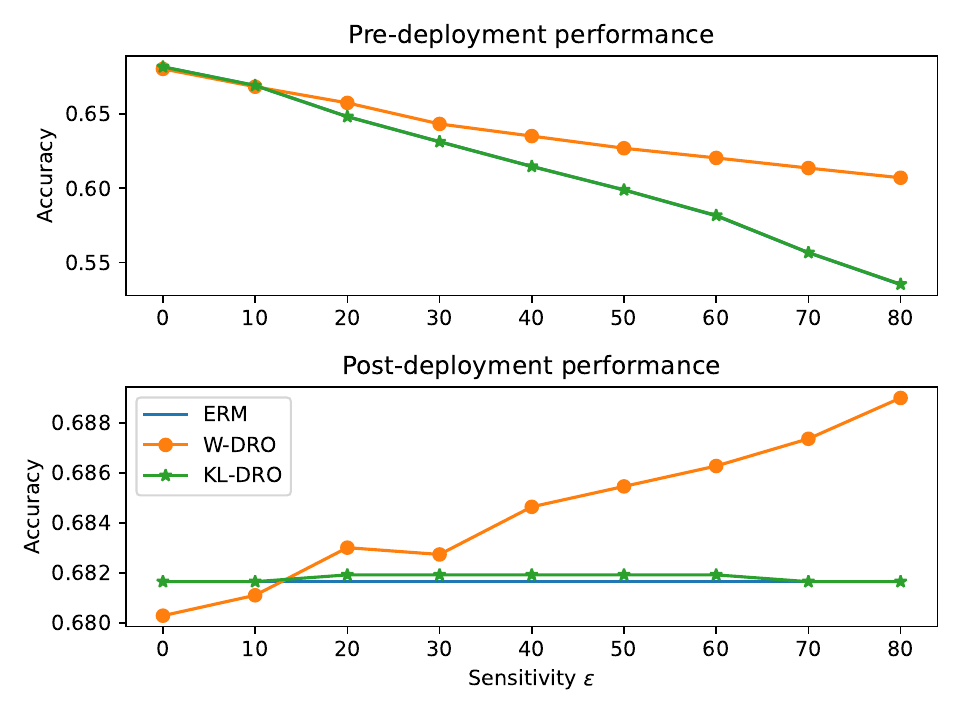}
    \caption{Pre-deployment (top) and post-deployment (bottom) accuracies of empirical risk minimization (ERM), Wasserstein DRO (W-DRO, Algorithm~\ref{alg:risk minimization}), and KL-DRO, across varying sensitivity levels $\varepsilon$. ``Pre'' and ``Post'' denote before and after the distribution adapts to the model deployments.  }
    \label{fig:compare_acc_algs}
\end{figure}

\begin{figure}
    \centering
    \includegraphics[width=\linewidth]{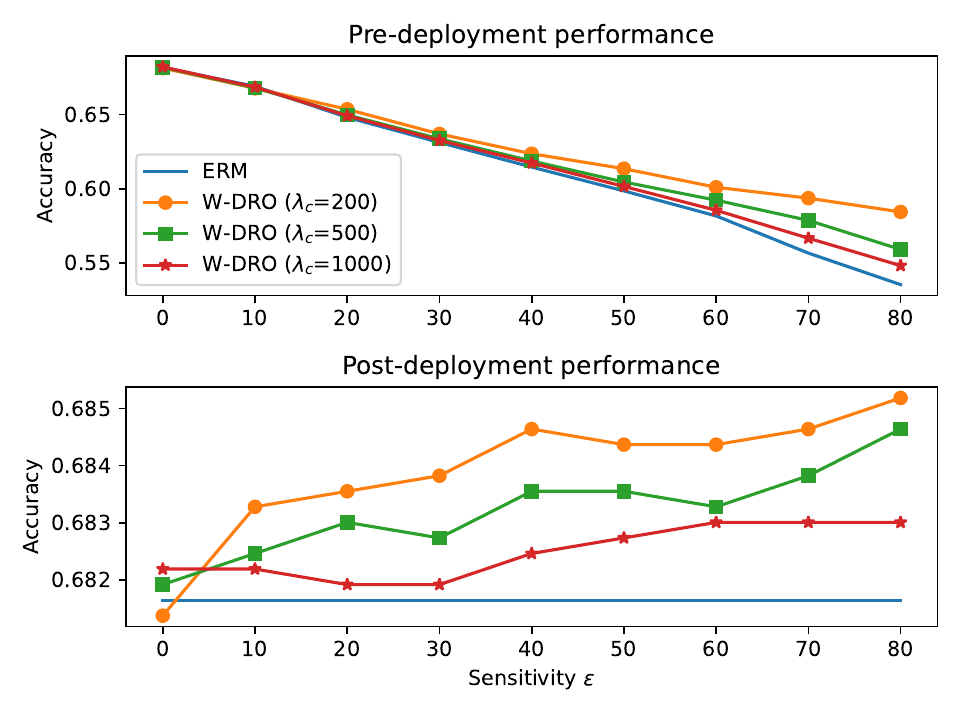}
    \caption{Pre-deployment (top) and post-deployment (bottom) accuracies under varying sensitivity $\varepsilon$. The parameter $\lambda_c$ dictates the transportation cost penalty in the W-DRO formulation, where a smaller $\lambda_c$ equates to a larger robust radius, and vice versa. 
}
    \label{fig:compare radius}
\end{figure}


\section{Conclusion}\label{sec:conclusion}
In this paper, we introduced a distributionally robust performative prediction framework grounded in empirical data distributions. To address the inherent discrepancy between the empirical and true distributions, we employed a decision-dependent Wasserstein ambiguity set centered on the empirical distribution to encompass the true distribution. By leveraging Lagrangian relaxation and strong duality, we transformed the robust objective into a computationally tractable form.
Within this framework, we defined robust performative stable points and the robust performative optimum. We developed DR-RRM and DR-RGD algorithms to iteratively find the stable points. We proved that under standard regularity conditions, both algorithms converge to a stable point in the exact setting, or to a bounded neighborhood in the presence of approximation errors. Furthermore, we established theoretical bounds on the parameter distance and risk gap between the robust stable points and the global performative optimum. These results show that the distributionally robust performative prediction framework effectively mitigates the impact of strategic manipulation while maintaining robust performance under distributional uncertainty.
Future work will explore extending this framework to multi-agent game-theoretic settings.

\section*{Appendix}\label{sec:appendix}
\noindent\textit{Proof of Lemma~\ref{lemma:phi Lipschitzs}.}
In the following, we will show that the gradient $\nabla_\theta\phi(\theta, \xi, \zeta)$ is Lipschitz continuous with respect to the arguments $\theta$, $\xi$, and $\zeta$, with Lipschitz constants $L_{\theta\theta}^\phi$, $L_{\theta\xi}^\phi$, and $L_{\theta\zeta}^\phi$, respectively. 

To begin with, for $\theta,\theta' \in \Theta$, we have that 
\begin{align*} 
&\|\nabla_\theta\phi(\theta,\xi,\zeta)-\nabla_\theta\phi(\theta',\xi,\zeta)\| \nonumber \\ 
=& \| \nabla_\theta l(\theta,\zeta)-\nabla_\theta l(\theta',\zeta)-\big(\nabla_\theta\lambda(\theta)-\nabla_\theta\lambda(\theta') \big) c(\xi,\zeta)\| \nonumber \\ 
\le &  \| \nabla_\theta l(\theta,\zeta)-\nabla_\theta l(\theta',\zeta)\| + \| \nabla_\theta\lambda(\theta)-\nabla_\theta\lambda(\theta')  \|\| c(\xi,\zeta)\|  \nonumber \\
\le& L_{\theta\theta}^\phi \|\theta-\theta'\|,
\end{align*}
where $ L_{\theta\theta}^\phi = L_{\theta\theta}+H_\lambda D_\xi^2$. The last inequality follows from the Lipschitzness of the functions $\nabla_\theta l(\theta,\zeta)$ and $\nabla_\theta\lambda(\theta)$, as in Assumptions~\ref{assumption:loss function} and \ref{assumption:lambda}, along with $\|c(\xi,\zeta)\| = \|\xi-\zeta\|^2\le D_\xi^2$, with $D_\xi$ being the diameter of the set $\Xi$.

Similarly, for $\xi,\xi' \in \Xi$, we have that 
\begin{align}\label{eq:L_phi_theta_xi}
&\|\nabla_\theta\phi(\theta,\xi,\zeta)-\nabla_\theta\phi(\theta,\xi',\zeta)\| \nonumber  \\ 
&= \| \nabla_\theta l(\theta,\zeta)-\nabla_\theta l(\theta,\zeta)-\nabla_\theta\lambda(\theta)(  c(\xi,\zeta)- c(\xi',\zeta)\big) \| \nonumber  \\ 
&\le \|\nabla_\theta\lambda(\theta)\| \| c(\xi,\zeta)- c(\xi',\zeta)\|
\nonumber \\ 
&\le
L_{\theta\xi}^\phi \|\xi-\xi'\| ,
\end{align}
where $ L_{\theta\xi}^\phi = 2L_\lambda D_\xi$. The last inequality holds by 
\begin{align}\label{eq:c difference}
&\|c(\xi,\zeta)-c(\xi',\zeta)\| \nonumber\\ 
&= \big\| \|\xi-\zeta\|^2-\|\xi'-\zeta\|^2    \big\| \nonumber \\
& = \big\|(\xi-\zeta)^\top(\xi-\zeta)-(\xi'-\zeta)^\top(\xi'-\zeta) \big\| \nonumber \\ 
& =  \big\|\xi^\top \xi - 2\xi^\top \zeta    - {\xi'}^\top \xi'
 + 2{\xi'}^\top \zeta  \big\|   \nonumber \\ 
& =  \big\|(\xi -\xi')^\top (\xi+\xi'-2\zeta) \big\|  \le 2D_\xi \| \xi-\xi'\|, 
\end{align} 
where the inequality holds by the fact that $D_\xi$ is the diameter of the set $\Xi$.

For $\zeta,\zeta' \in \Xi$, we have that 
\begin{align*}
&\|\nabla_\theta\phi(\theta,\xi,\zeta)-\nabla_\theta\phi(\theta,\xi,\zeta')\| \nonumber \\ 
&= \big\| \nabla_\theta l(\theta,\zeta)-\nabla_\theta l(\theta,\zeta')-\nabla_\theta\lambda(\theta)\big( c(\xi,\zeta)-c(\xi,\zeta') \big)\big\| \nonumber \\ 
&\le  \| \nabla_\theta l(\theta,\zeta)-\nabla_\theta l(\theta,\zeta')\| + \|\nabla_\theta\lambda(\theta)\| \| c(\xi,\zeta)-c(\xi,\zeta')  \|   \nonumber \\
&\le L_{\theta\zeta}\|\zeta-\zeta'\| + 2L_{\lambda}D_\xi \|\zeta-\zeta'\|
=
L_{\theta\zeta}^\phi \|\zeta-\zeta'\| ,
\end{align*}
where $L_{\theta\zeta}^\phi = L_{\theta\zeta}+2L_{\lambda}D_\xi$. 
The second inequality holds by $ \| c(\xi,\zeta)-c(\xi,\zeta')  \| \le 2D_\xi \|\zeta-\zeta'\| $, which follows the derivation of \eqref{eq:L_phi_theta_xi}. 

We next show that gradient mapping $\nabla_\zeta\phi(\theta, \xi, \zeta)$ is Lipschitz continuous with respect to $\theta$, with Lipschitz constant $L_{\zeta\theta}^\phi$. 
For $\theta,\theta' \in \Theta$, we have that 
\begin{align*} 
&\|\nabla_\zeta\phi(\theta,\xi,\zeta)-\nabla_\zeta\phi(\theta',\xi,\zeta)\| \nonumber \\ 
&= \| \nabla_\zeta l(\theta,\zeta)-\nabla_\zeta l(\theta',\zeta)-2\big(\lambda(\theta)-\lambda(\theta') \big)(\zeta-\xi)\| \nonumber \\ 
&\le  \| \nabla_\zeta l(\theta,\zeta)-\nabla_\zeta l(\theta',\zeta)\|+ 2\| \lambda(\theta)-\lambda(\theta') \|\|\zeta-\xi\| \nonumber \\
&\le L_{\zeta\theta}^\phi \|\theta-\theta'\|, 
\end{align*}
where $L_{\zeta\theta}^\phi = L_{\zeta\theta} + 2L_\lambda D_\xi $. The equality follows from $\nabla_\zeta c(\xi,\zeta)=\nabla_\zeta \|\xi-\zeta\|^2 = 2(\zeta-\xi) $.
The second inequality follows from the Lipschitzness of 
$\nabla_\zeta l(\cdot,\zeta)$ and $\lambda(\cdot)$, as in 
Assumptions~\ref{assumption:loss function} and \ref{assumption:lambda}. 

Finally, we show that the gradient mapping $\nabla_\zeta\phi(\theta, \xi, \zeta)$ is Lipschitz continuous with respect to $\xi$, with Lipschitz constant $L_{\zeta\xi}^\phi$. 
For $\xi,\xi' \in \Xi$, we have that 
\begin{align*} 
&\|\nabla_\zeta\phi(\theta,\xi,\zeta)-\nabla_\zeta\phi(\theta,\xi',\zeta)\| \nonumber \\ 
& = \| \nabla_\zeta l(\theta,\zeta)-\nabla_\zeta l(\theta,\zeta)-2\lambda(\theta) (\zeta-\xi-\zeta+\xi')\| \nonumber \\ 
&\le   2\| \lambda(\theta)\|  \|\xi-\xi'\|  
\le L_{\zeta\xi}^\phi \|\xi-\xi'\| ,
\end{align*}
where $L_{\zeta\xi}^\phi = 2\lambda_{\max}$. The  equality follows from $\nabla_\zeta c(\xi,\zeta)=2(\zeta-\xi)$. The last inequality follows from $\lambda(\theta)\le \lambda_{\max}$, as stated in Assumption~\ref{assumption:lambda}.
\hfill \qed 

\noindent \textit{Proof of Lemma~\ref{lemma:f Lipschitz}}. 
Since $\phi(\theta,\xi,\zeta)$ is $\mu$-strongly concave in $\zeta$, for $\theta_1,\theta_2 \in \Theta$ and all $\xi\in\Xi$, we have that 
\begin{align}\label{eq:strongly concave inequality1}
&\frac{\mu}{2}\|\zeta^\ast(\theta_1,\xi) - \zeta^\ast(\theta_2,\xi)\|^2 \nonumber \\ 
&\le \phi(\theta_2,\xi,\zeta^\ast(\theta_2,\xi)) - \phi(\theta_2,\xi,\zeta^\ast(\theta_1,\xi)),
\end{align} and 
\begin{align}\label{eq:strongly concave inequality2}
&\phi(\theta_2,\xi,\zeta^\ast(\theta_2,\xi))  \le \phi(\theta_2,\xi,\zeta^\ast(\theta_1,\xi))\nonumber\\
&\hspace{1em} + \nabla_\zeta\phi(\theta_2,\xi,\zeta^\ast(\theta_1,\xi))^\top(\zeta^\ast(\theta_2,\xi)
-\zeta^\ast(\theta_1,\xi)) \nonumber \\
& \hspace{1em}- \frac{\mu}{2}\|\zeta^\ast(\theta_1,\xi)-\zeta^\ast(\theta_2,\xi)\|^2,    
\end{align}
where the first inequality follows from the optimality of $\zeta^\ast(\theta_2,\xi)$ and  $\nabla_\zeta \phi(\theta_2,\xi,\zeta^\ast(\theta_2,\xi))^\top (\zeta^\ast(\theta_1,\xi) - \zeta^\ast(\theta_2,\xi)) \le 0$.
Combining inequalities \eqref{eq:strongly concave inequality1} and \eqref{eq:strongly concave inequality2}, we obtain 
\begin{align}\label{eq:xi^2 le theta xi}
&\mu\|\zeta^\ast(\theta_1,\xi)-\zeta^\ast(\theta_2,\xi)\|^2 \nonumber \\
&\le \nabla_\zeta\phi(\theta_2,\xi,\zeta^\ast(\theta_1,\xi))^\top(\zeta^\ast(\theta_2,\xi)-\zeta^\ast(\theta_1,\xi) ) \nonumber \\
& \le \big(\nabla_\zeta\phi(\theta_2,\xi,\zeta^\ast(\theta_1,\xi))-\nabla_\zeta\phi(\theta_1,\xi,\zeta^\ast(\theta_1,\xi))\big)^\top\nonumber \\
&\hspace{1em}(\zeta^\ast(\theta_2,\xi)-\zeta^\ast(\theta_1,\xi) ) \nonumber \\ 
& \le \| \nabla_\zeta\phi(\theta_2,\xi,\zeta^\ast(\theta_1,\xi))-\nabla_\zeta\phi(\theta_1,\xi,\zeta^\ast(\theta_1,\xi))\|\nonumber \\
&\hspace{1em}\|\zeta^\ast(\theta_1,\xi) - \zeta^\ast(\theta_2,\xi)\| \nonumber \\
& \le L_{\zeta\theta}^\phi \|\theta_1-\theta_2\|\|\zeta^\ast(\theta_1,\xi) - \zeta^\ast(\theta_2,\xi)\| ,
\end{align}
where the second inequality holds by $\nabla_\zeta \phi(\theta_1,\xi,\zeta^\ast(\theta_1,\xi))^\top (\zeta^\ast(\theta_2,\xi) - \zeta^\ast(\theta_1,\xi)) \le 0$. The last inequality follows from the Lipschitzness of $\nabla_\zeta\phi$ in $\theta$, as in Lemma~\ref{lemma:phi Lipschitzs}.  
By \eqref{eq:xi^2 le theta xi}, we obtain  
\begin{equation}\label{eq:zeta(theta) Lipschitz}
\|\zeta^\ast(\theta_1,\xi)-\zeta^\ast(\theta_2,\xi)\| \le \frac{L_{\zeta\theta}^\phi}{\mu}\|\theta_1-\theta_2\|.   
\end{equation}

As stated in Assumptions~\ref{assumption:loss function} and \ref{assumption:lambda}, the function $\phi(\theta,\xi,\zeta)$ is differentiable and strongly concave in $\zeta$. 
Then, by Danskin's Theorem, we have that $\nabla_\theta f(\theta,\xi) =  \nabla_\theta \phi(\theta,\xi,\zeta^\ast(\theta,\xi))   $, where $\zeta^\ast(\theta,\xi) \in {\arg\max}_{\zeta}\phi(\theta,\xi,\zeta)$ for  $(\theta,\xi)\in \Theta \times \Xi$.   
Additionally, by the strongly concavity of $\phi(\theta,\xi,\cdot)$, $\zeta^\ast(\theta,\xi)$ is unique for each $(\theta,\xi)$. 
Using these results, we next bound $\|\nabla_\theta f(\theta,\xi)-\nabla_\theta f(\theta',\xi)\|$. For all $\theta_1,\theta_2 \in \Theta$ and $\xi \in \Xi$, we have that  
\begin{align}\label{eq:Lipschitz_f_theta_theta}
&\|\nabla_\theta f(\theta_1,\xi) - \nabla_\theta f(\theta_2,\xi)\| \nonumber \\ 
&= \| \nabla_\theta\phi(\theta_1,\xi, \zeta^\ast(\theta_1,\xi))-\nabla_\theta\phi(\theta_2,\xi, \zeta^\ast(\theta_2,\xi)) \nonumber \|\\ 
&\le  \| \nabla_\theta \phi(\theta_1,\xi, \zeta^\ast(\theta_1,\xi)) -  \nabla_\theta \phi(\theta_2,\xi, \zeta^\ast(\theta_1,\xi))\| \nonumber \\ 
&\hspace{1em}+\|\nabla_\theta \phi(\theta_2,\xi,\zeta^\ast(\theta_1,\xi))-\nabla_\theta \phi(\theta_2,\xi,\zeta^\ast(\theta_2,\xi))\|\nonumber \\ 
 &\le L_{\theta\theta}^\phi\|\theta_1-\theta_2\| + L_{\theta\zeta}^\phi\|\zeta^\ast(\theta_1,\xi)-\zeta^\ast(\theta_2,\xi)\| \nonumber\\
 & \le \big(L_{\theta\theta}^\phi + L_{\theta\zeta}^\phi L_{\zeta\theta}^\phi / \mu\big)\| \theta_1-\theta_2\|  =  L_{\theta\theta}^f\| \theta_1-\theta_2\|, 
\end{align}
where $L_{\theta\theta}^f=  L_{\theta\theta}^\phi + L_{\theta\zeta}^\phi L_{\zeta\theta}^\phi / \mu $. 
The second inequality follows from the Lipschitzness of $\nabla_\theta \phi(\theta,\xi,\zeta)$ with respect to $\theta$ and $\zeta$. The third inequality follows from substituting \eqref{eq:zeta(theta) Lipschitz} into \eqref{eq:Lipschitz_f_theta_theta}.

We next prove the second claim. By following the derivation of \eqref{eq:strongly concave inequality1}--\eqref{eq:xi^2 le theta xi},  and by the Lipschitzness of $\nabla_\zeta \phi$ in $\xi$, as in Lemma~\ref{lemma:phi Lipschitzs},  we have that, for $\xi_1, \xi_2 \in \Xi$, and for a fixed but arbitrary $\theta$,
\begin{align*}
&\mu\|\zeta^\ast(\theta,\xi_1)-\zeta^\ast(\theta,\xi_2)\|^2 \nonumber \\ 
&\le    
  L_{\zeta\xi}^\phi\|\xi_1-\xi_2\|\|\zeta^\ast(\theta,\xi_1) - \zeta^\ast(\theta,\xi_2)\| .
\end{align*}
Therefore, we obtain 
\begin{equation}\label{eq:zeta(xi) Lipschitz}
\|\zeta^\ast(\theta,\xi_1)-\zeta^\ast(\theta,\xi_2)\|   \le \frac{L_{\zeta\xi}^\phi}{\mu}\|\xi_1-\xi_2\| .
\end{equation}
We next bound $\|\nabla_\theta f(\theta,\xi_1)-\nabla_\theta f(\theta,\xi_2)\|$. For $\theta \in \Theta$, and $\xi_1,\xi_2\in\Xi$, we have that   
\begin{align}\label{eq:Lipschitz_f_theta_x}
&\|\nabla_\theta f(\theta,\xi_1) - \nabla_\theta f(\theta,\xi_2)\| \nonumber \\ 
& = \| \nabla_\theta\phi(\theta,\xi_1,\zeta^\ast(\theta,\xi_1))-\nabla_\theta\phi(\theta,\xi_2,\zeta^\ast(\theta,\xi_2)) \nonumber \|\\ 
& \le   \| \nabla_\theta\phi(\theta,\xi_1,\zeta^\ast(\theta,\xi_1))-\nabla_\theta\phi(\theta,\xi_1,\zeta^\ast(\theta,\xi_2))\| \nonumber\\ 
&\hspace{1em} +\|\nabla_\theta\phi(\theta,\xi_1,\zeta^\ast(\theta,\xi_2))-\nabla_\theta\phi(\theta,\xi_2,\zeta^\ast(\theta,\xi_2)) \nonumber \| \nonumber \\ 
& \le  L_{\theta\zeta}^\phi\|\zeta^\ast(\theta,\xi_1) - \zeta^\ast(\theta,\xi_2)\| +L_{\theta\xi}^\phi\|\xi_1  - \xi_2 \| \nonumber \\ 
&  \le L_{\theta\xi}^f \|\xi_1-\xi_2\|,
\end{align}
where $L_{\theta\xi}^f =L_{\theta\xi}^\phi 
+L_{\theta\zeta}^\phi L_{\zeta\xi}^\phi/\mu   $. The second inequality follows from the Lipschitzness of $\nabla_\theta\phi(\theta,\zeta,\xi)$ in $\zeta$ and $\xi$, as in \eqref{eq:phi lip theta_zeta} and \eqref{eq:phi lip theta_xi}, respectively. The last inequality follows from substituting \eqref{eq:zeta(xi) Lipschitz} into \eqref{eq:Lipschitz_f_theta_x}.
\hfill \qed

\noindent\textit{Proof of Theorem~\ref{theo:risk minimization}.} 
In order to investigate the convergence of Algorithm~\ref{alg:risk minimization}, we will first show that in the absence of error, the risk minimization is a contraction mapping, and then examine the effect of the approximation error.
For $\theta,\theta',x\in\Theta$, define auxiliary functions 
\begin{align}\label{eq:rm opt obj}
F_\theta(x)&: = \mathbb{E}_{\xi\sim\hat{\mathbb{P}}(\theta)}\big[ \phi\big(x,\xi,\zeta^\ast(\theta,\xi)\big) \big], \nonumber  \\
F_{\theta'}(x)&: = \mathbb{E}_{\xi\sim\hat{\mathbb{P}}(\theta')}\big[\phi\big(x,\xi,\zeta^\ast(\theta',\xi)\big)\big],
\end{align}
which represent the objective function of the risk minimization \eqref{eq:Alg risk minimization} in the exact setting.
These functions decouple the dependence on the distribution induced by the deployed parameter $\theta$ from the dependence on the candidate model parameter $x$.
Moreover, denote  
\begin{align}\label{eq:rm mapping}
G(\theta) = \mathop{\arg\min}_{x \in \Theta}  F_\theta(x),  \quad
G(\theta')  = \mathop{\arg\min}_{x \in \Theta}  F_{\theta'}(x),
\end{align}
which represent the risk minimization in the exact setting. 
Since the function $\phi(x,\xi,\zeta)$ is $\gamma$-strongly convex in $x$, so are 
$F_\theta(x)$ and $F_{\theta'}(x)$. Then,  we have that 
\begin{align}\label{eq:J_eta strongly convex 1}
& F_\theta(G(\theta))- F_\theta(G(\theta')) \nonumber \\
&\!\ge\! \big(G(\theta)-G(\theta') \big)^\top \nabla F_\theta(G(\theta'))\!+\!\frac{\gamma}{2}\|G(\theta)-G(\theta')\|^2  ,  
\end{align}
where $\nabla F_\theta(x)$ is the gradient of $F_\theta$ with respect to $x$. Additionally, we have that    
\begin{align}\label{eq:J_eta strongly convex 2}    
\hspace{-1em} & F_\theta(G(\theta'))- F_\theta(G(\theta))  \nonumber \\ 
\hspace{-1em}&\ge  \big(G(\theta')-G(\theta) \big)^\top\nabla F_\theta(G(\theta))+\frac{\gamma}{2}\|G(\theta')-G(\theta)\|^2 \nonumber\\ 
&\ge \frac{\gamma}{2}\|G(\theta')-G(\theta)\|^2 ,
\end{align}
 where the second inequality follows from the first-order optimality condition $\big(G(\theta')-G(\theta) \big)^\top\nabla F_\theta(G(\theta)) \ge 0$.  
Combining \eqref{eq:J_eta strongly convex 1} and \eqref{eq:J_eta strongly convex 2}, we have that 
\begin{flalign*} 
& \gamma\|G(\theta')-G(\theta)\|^2 \nonumber \\ 
& \le - \big(G(\theta)-G(\theta') \big)^\top\nabla F_\theta(G(\theta'))\nonumber\\
& \le \big(G(\theta)-G(\theta') \big)^\top\big(\nabla F_{\theta'}(G(\theta'))- \nabla F_\theta(G(\theta'))\big) \nonumber \\
& \le \big\|(G(\theta)-G(\theta') \big\| \big\|\nabla F_{\theta'}(G(\theta'))- \nabla F_\theta(G(\theta'))\big\|, 
\end{flalign*}
where the second inequality follows from the first-order optimality condition $\big(G(\theta)-G(\theta') \big)^\top\nabla F_{\theta'}(G(\theta'))\ge0$. 
Therefore, we obtain 
\begin{equation}\label{eq:G norm 3}
\gamma\|G(\theta')-G(\theta)\|  \le\|\nabla F_{\theta'}(G(\theta'))- \nabla F_\theta(G(\theta'))\big\|.   
\end{equation}
We next bound the right-hand side. 
Denote $h_\alpha (x,\xi): =\nabla_x \phi(x,\xi,\zeta^\ast(\alpha,\xi))$. Note that $\xi$ and $\zeta^\ast(\theta,\xi)$ both depend on the deployed parameter $\theta$, and are independent of the optimization variable $ x$. Under the smoothness assumptions, we may interchange differentiation and expectation to obtain
\begin{align*}
&\nabla F_\theta (x)  = \mathbb{E}_{\xi\sim\hat{\mathbb{P}}(\theta)}[\nabla_x\phi(x,\xi,\zeta^\ast(\theta,\xi))] = \mathbb{E}_{\xi\sim\hat{\mathbb{P}}(\theta)}[h_{\theta}(x,\xi)].    
\end{align*}
Similarly, we have $\nabla F_{\theta'}(G(\theta'))=\mathbb{E}_{\xi\sim\hat{\mathbb{P}}(\theta')}[h_{\theta'} (G(\theta'),\xi)]$  and  $\nabla F_{\theta}(G(\theta'))=\mathbb{E}_{\xi\sim\hat{\mathbb{P}}(\theta)}[h_{\theta} (G(\theta'),\xi)]$.  
By adding and subtracting $\mathbb{E}_{\xi\sim\hat{\mathbb{P}}(\theta)}[h_{\theta'} (G(\theta'),\xi)]$,
we have  
\begin{align}\label{eq:G norm 4}
&\|\nabla F_{\theta'}(G(\theta'))- \nabla F_\theta(G(\theta'))\big\| \nonumber \\ 
&\le \big\| \mathbb{E}_{\xi\sim\hat{\mathbb{P}}(\theta')}[h_{\theta'} (G(\theta'),\xi)] - \mathbb{E}_{\xi\sim\hat{\mathbb{P}}(\theta)}[h_{\theta'} (G(\theta'),\xi)]  \big\| \nonumber \\ 
& \hspace{1em}+  \mathbb{E}_{\xi\sim\hat{\mathbb{P}}(\theta)}\big[\|h_{\theta'} (G(\theta'),\xi) -h_{\theta} (G(\theta'),\xi)\|\big] . 
\end{align}
For the first term, using the same argument as in Lemma~\ref{lemma:f Lipschitz}, the mapping $h_{\theta'}(G(\theta'),\xi)$ is $L_{\theta\xi}^f$ in $\xi$. Hence, by Lemma~\ref{lemma:wasserstein} and $\varepsilon$-sensitivity of $\hat{\mathbb{P}}(\cdot)$, we have 
\begin{align}\label{eq: first term}
&    \big\| \mathbb{E}_{\xi\sim\hat{\mathbb{P}}(\theta')}[h_{\theta'} (G(\theta'),\xi)] - \mathbb{E}_{\xi\sim\hat{\mathbb{P}}(\theta)}[h_{\theta'} (G(\theta'),\xi)]  \big\| \nonumber \\ 
&\le L_{\theta\xi}^f W_1\big(\hat{\mathbb{P}}(\theta'),\hat{\mathbb{P}}(\theta)\big) \le L_{\theta\xi}^f \varepsilon \| \theta-\theta'\|. 
\end{align}
In the second term of \eqref{eq:G norm 4}, the adversarial
maximizer changes from $ \zeta^*(\theta,\xi)$ to
$\zeta^*(\theta',\xi)$. By the Lipschitz continuity of $\nabla_x \phi$ with respect to $\zeta$, together with the sensitivity of the maximizer map $\zeta^\ast(\theta,\xi)$, we have\begin{align}\label{dphi theta}
&\| h_{\theta'} (G(\theta'),\xi) -h_{\theta} (G(\theta'),\xi) \| \nonumber\\
&\le \| \nabla_x \phi(G(\theta'),\xi,\zeta^\ast(\theta',\xi))-\nabla_x \phi(G(\theta'),\xi,\zeta^\ast(\theta,\xi))\| \nonumber\\
&\le L_{\theta\zeta}^\phi \|\zeta^\ast(\theta,\xi)) - \zeta^\ast(\theta',\xi))\| \le \frac{L_{\theta\zeta}^\phi L_{\zeta\theta}^\phi}{\mu}\|\theta-\theta'\|.
\end{align}
Substituting \eqref{eq:G norm 4}--\eqref{dphi theta} into \eqref{eq:G norm 3} yields 
\begin{align}\label{eq:gamma Gs difference 2}
 &  \|G(\theta')-G(\theta)\|    \le \kappa_{\rm rm}  \|\theta-\theta'\|,
\end{align} 
Thus, when $\kappa_{\rm rm} = \big( L_{\theta\xi}^f \varepsilon  +  L_{\theta\zeta}^\phi L_{\zeta\theta}^\phi/\mu\big)/\gamma<1$, $G(\cdot)$ is a contraction. Namely, in the absence of the approximation error, the risk minimization converges to a unique stable point $\theta_s$.

We next examine the risk minimization with the approximation error. 
Similarly, following the way of defining \eqref{eq:rm opt obj}, we define the minimization objective in \eqref{eq:Alg risk minimization} as   
\begin{equation*}
F_\theta^\epsilon (x): = \mathbb{E}_{\xi\sim\hat{\mathbb{P}}(\theta)} \big[ \phi\big(x,\xi,z(\theta,\xi) \big) \big].
\end{equation*}   
Therefore, we have that
\begin{align}\label{eq:zeta difference}
 &\| \nabla F_\theta(x ) -\nabla F_\theta^\epsilon(x) \| \nonumber \\ 
 &= \|\mathbb{E}_{\xi\sim\hat{\mathbb{P}}(\theta)}[\nabla_x\phi(x,\xi,z(\theta,\xi))-\nabla_x\phi(x,\xi,\zeta^\ast(\theta,\xi)) ]\| \nonumber \\
 &\le \mathbb{E}_{\xi\sim\hat{\mathbb{P}}(\theta)}[\|\nabla_x\phi(x,\xi,z(\theta,\xi))-\nabla_x\phi(x,\xi,\zeta^\ast(\theta,\xi)) \|] \nonumber \\
&\le \mathbb{E}_{\xi\sim\hat{\mathbb{P}}(\theta)}\big[L_{\theta\zeta}^\phi \|z(\theta,\xi)-\zeta^\ast(\theta,\xi)\|\big] \le L_{\theta\zeta}^\phi\epsilon,  \end{align}
where the second inequality follows from  
$\nabla_\theta\phi(\theta,\xi,\zeta)$  is $L_{\theta\zeta}^\phi$-lipschitz in $\zeta$. The last inequality follows from 
the definition of $z(\theta,\xi)$. Define the right-hand side of \eqref{eq:Alg risk minimization} as $G^\epsilon(\cdot)$. Thus, we have $\theta_{t+1} = G^\epsilon (\theta_t)$ and 
\begin{equation}\label{eq:rm mapping epsilon}
G^\epsilon(\theta) = {\arg\min}_{x \in \Theta} F_\theta^\epsilon(x).
\end{equation} 
By the strong convexity of $f(\cdot,\xi)$, we have that the first-order condition
\begin{align}\label{eq: G aprox bound 1}
&\gamma\|G(\theta)-G^\epsilon(\theta)\|^2 \nonumber \\
&\le \langle \nabla F_\theta(G^\epsilon(\theta)) - \nabla F_\theta(G(\theta)), G^\epsilon(\theta)-G (\theta) \rangle  
\end{align}
By the definitions of $G(\theta)$ and $G^\epsilon(\theta)$, as in \eqref{eq:rm mapping}  and \eqref{eq:rm mapping epsilon}, we have   
\begin{align*}
&\big\langle\nabla F_\theta (G(\theta)), G^\epsilon(\theta)- G(\theta)\big\rangle \ge 0, \nonumber \\
& \big\langle\nabla F_\theta^\epsilon (G^\epsilon(\theta)), G(\theta)- G^\epsilon(\theta)\big\rangle \ge 0. 
\end{align*}
Substituting these inequalities into \eqref{eq: G aprox bound 1}, we obtain 
\begin{align}\label{eq: G aprox bound 2}
&\gamma\|G(\theta)-G^\epsilon(\theta)\|^2 \nonumber \\
&\le \big(\nabla F_\theta(G^\epsilon(\theta)) - \nabla F_\theta^\epsilon(G^\epsilon(\theta)\big)^\top
\big(G^\epsilon(\theta) - G(\theta)\big) \nonumber \\ 
&\le \| \nabla F_\theta(G^\epsilon(\theta)) - \nabla F_\theta^\epsilon(G^\epsilon(\theta))\| \| G^\epsilon(\theta) - G(\theta)\| \nonumber \\ 
&\le L_{\theta\zeta}^\phi \epsilon \| G^\epsilon(\theta) - G(\theta)\|,
\end{align}
where the last inequality follows from substituting \eqref{eq:zeta difference} into \eqref{eq: G aprox bound 2}. Thus, we obtain 
\begin{equation}\label{eq:G aprox bound}
 \|G(\theta)-G^\epsilon(\theta)\| \le \frac{L_{\theta\zeta}^\phi \epsilon}{\gamma}   .
\end{equation}
We now combine the contraction property of the exact update map $G(\cdot)$
with the perturbation bound $\epsilon$.
Recall that the exact stable point satisfies
$\theta_s = G(\theta_s),$ 
whereas the iterate generated by Algorithm 1 satisfies $\theta_{t+1}=G^\epsilon(\theta_t).
$
By adding and subtracting $G(\theta_t)$, we obtain 
\begin{align*}
&\|\theta_{t+1} - \theta_s\| =  \| G^\epsilon(\theta_t) -  G (\theta_s)  \|   \nonumber \\ 
& \le \| G^\epsilon(\theta_t) -  G (\theta_t)  \|   + \| G (\theta_t)  - G (\theta_s)  \| \nonumber \\ 
&\le \kappa_{\rm rm} \|\theta_t-\theta_s\| + \frac{L_{\theta\zeta}^\phi \epsilon}{\gamma} ,
\end{align*} 
where the second inequality follows from \eqref{eq:gamma Gs difference 2} and \eqref{eq:G aprox bound}.
By the infinite geometric series sum, we obtain the results.
\hfill \qed
 
\noindent\textit{Proof of Theorem~\ref{theo:gradient descent}:}
Following the proof of Theorem~\ref{theo:risk minimization}, we first investigate the convergence of Algorithm~\ref{alg:gradient descent} in the exact setting, and then examine its performance taking into account the approximation error. In the absence of approximation error, the gradient estimate is  \begin{equation}\label{eq:gradient estimate exact} 
\tilde{g} (\theta_t) = \mathbb{E}_{\xi \sim \hat{\mathbb{P}}(\theta_t)} \big[\nabla_\theta  \phi\big(\theta_t,\xi ,\zeta^\ast(\theta_t,\xi )\big)\big].
\end{equation}
Note that by Danskin's Theorem, we have $\nabla_\theta f(\theta,\xi) = \nabla_\theta \phi (\theta,\xi,\zeta^\ast(\theta,\xi))$. Therefore, $$\tilde{g}(\theta_t) = \mathbb{E}_{\xi\sim \hat{\mathbb{P}}(\theta_t)}[\nabla_{\theta}f(\theta_t,\xi)],$$ which is the gradient of the exact robust loss $f$. 
Given \eqref{eq:gradient estimate exact}, in the exact setting, the decision is updated according to 
\begin{equation}\label{eq:gradient descent exact} 
\tilde{\theta}_{t+1} = G_{\rm gd}(\tilde{\theta}_t)= {\rm Proj}_{\Theta}\big(\tilde{\theta}_t - \eta \tilde{g}(\tilde{\theta}_t)\big)  ,
\end{equation}
with initial value $\tilde{\theta}_0 = \theta_0$.
Denote the gradient descent operation with the approximation error, i.e., right-hand side of \eqref{eq:repeat gradient descent} as $G_{\rm gd
}^\epsilon$, so we have $\theta_{t+1}= G_{\rm gd}^\epsilon(\theta_t)$. 
Since the projection onto a convex set is non-expansive, we omit the projection operator ${\rm Proj}$ when analyzing the contraction property of $G_{\rm gd}(\cdot)$ and $G_{\rm gd}^\epsilon(\cdot)$.

For $\theta,\theta' \in \Theta$, we have that 
\begin{align}\label{eq:G_gd diff}
\hspace{-0.5em}G_{\rm gd}(\theta)-G_{\rm gd}(\theta') =  \theta - \eta \tilde{g}(\theta) - \big(\theta' -  \eta \tilde{g}(\theta')\big).
\end{align}
By taking the square norm on both sides of \eqref{eq:G_gd diff}, we obtain 
\begin{align}\label{eq:G_gd contraction}
&\|G_{\rm gd}(\theta)-G_{\rm gd}(\theta')\|^2 \nonumber \\
&=  \big\|\theta - \eta  \tilde{g}(\theta)  - \big(\theta' -  \eta  \tilde{g}(\theta') \big)  \big\|^2  \nonumber \\ 
&= \|\theta-\theta'\|^2 - 2\eta(\theta-\theta')^\top\big( \tilde{g}(\theta)  -  \tilde{g}(\theta') \big) \nonumber \\
&\hspace{1em}+\eta^2\|\tilde{g}(\theta)  -  \tilde{g}(\theta') \|^2 \nonumber \\
&=  T_1 - 2\eta^\top T_2 + \eta^2 T_3,
\end{align}
where $T_1=\|\theta-\theta'\|^2$, $T_3 = \|\tilde{g}(\theta)  - \tilde{g}(\theta')\|^2$   and $T_2=(\theta-\theta')^\top\big(\tilde{g}(\theta) - \tilde{g}(\theta')\big)$.
We next bound $T_2$ and $T_3$, respectively.
For 
$T_2$, by the $\gamma$-strong convexity of $F_\theta(\cdot)$, we have that 
\begin{align}\label{eq:T21}
&T_2 = (\theta-\theta')^\top\big(  \mathbb{E}_{\xi \sim \hat{\mathbb{P}}(\theta)}[\nabla f(\theta,\xi)]
- \mathbb{E}_{\xi \sim \hat{\mathbb{P}}(\theta')}[\nabla f(\theta',\xi)]\big) \nonumber \\
&= (\theta-\theta')^\top\big(  \mathbb{E}_{\xi \sim \hat{\mathbb{P}}(\theta)}[\nabla f(\theta,\xi)]
- \mathbb{E}_{\xi \sim \hat{\mathbb{P}}(\theta)}[\nabla f(\theta',\xi)] \nonumber \\
&\hspace{1em}+\mathbb{E}_{\xi \sim \hat{\mathbb{P}}(\theta)}[\nabla f(\theta',\xi)] - \mathbb{E}_{\xi \sim \hat{\mathbb{P}}(\theta')}[\nabla f(\theta',\xi)]\big) \nonumber \\
 &\ge \gamma\|\theta-\theta'\|^2 \nonumber \\ 
 &\hspace{1em}- \| \theta-\theta'\| \big\|  \mathbb{E}_{\xi \sim \hat{\mathbb{P}}(\theta)}[\nabla f(\theta',\xi)] - \mathbb{E}_{\xi \sim \hat{\mathbb{P}}(\theta')}[\nabla f(\theta',\xi)]\big\| \nonumber \\ 
&\ge  \Big(\gamma-\varepsilon L_{\theta\xi}^f \Big)\|\theta-\theta'\|^2,
\end{align}
where the last inequality follows from the $L_{\theta\xi}^f$-Lipschitz continuity $\nabla_\theta f(\theta,\xi)$ in $\xi$, and $\varepsilon$-sensitivity of $\hat{\mathbb{P}}(\cdot)$.
To avoid confusion, we denote $\nabla f(\theta,\xi)  := \nabla_\theta f(\theta,\xi)  $.
We next bound $T_3$: 
\begin{align}\label{eq:T3 0}
&T_3 
= \Big(\big\|\mathbb{E}_{\xi \sim \hat{\mathbb{P}}(\theta)}[\nabla f(\theta,\xi)]- \mathbb{E}_{\xi \sim \hat{\mathbb{P}}(\theta)}[\nabla f(\theta',\xi)]\big\|\nonumber \\ 
&\hspace{1em}+\big\| \mathbb{E}_{\xi \sim \hat{\mathbb{P}}(\theta)}[\nabla f(\theta',\xi)]-\mathbb{E}_{\xi \sim \hat{\mathbb{P}}(\theta')}[\nabla f(\theta',\xi)]\big\|\Big)^2 \nonumber \\ 
&\le \Big(L_{\theta\theta}^f+ \varepsilon L_{\theta\xi}^f \Big)^2\| \theta-\theta'\|^2,
\end{align}
where the inequality follows from the Lipschitzness of $\nabla_\theta f $ in $\theta$, and by following  the derivations of \eqref{eq:T21}.
To simplify notation, we denote $\beta =L_{\theta\theta}^f+ \varepsilon L_{\theta\xi}^f $
Substituting \eqref{eq:T21} and \eqref{eq:T3 0} 
into \eqref{eq:G_gd contraction}, we obtain 
\begin{align}\label{eq:G_gd contract mapping}
\|G_{\rm gd}(\theta)-G_{\rm gd}(\theta')\|^2  \!  \le \! \big(\eta^2\beta^2\! + \!2\eta(\varepsilon L_{\theta\xi}^f\!-\!\gamma) \!+\!1\big) \|\theta-\theta'\|^2 .
\end{align} 
As defined in Theorem~\ref{theo:gradient descent}, $\kappa_{\rm gd} = \sqrt{\eta^2\beta^2+2\eta(\varepsilon L_{\theta\xi}^f-\gamma)+1}$. When $0<\eta < \frac{2(\gamma-\varepsilon L_{\theta\xi}^f)}{\beta^2} $, we have that $\kappa_{\rm gd}<1 $. 
Taking $\theta = \theta_t$ and $\theta' = \theta_s$  in \eqref{eq:G_gd contract mapping}, and using $G_{\rm gd}(\theta_s) = \theta_s$, we obtain 
\begin{equation}\label{eq:contration exact}
\|\theta_{t+1}-\theta_s\| \le \kappa_{\rm gd} \|\theta_t-\theta_s\|.
\end{equation}
This means that $G_{\rm gd}(\cdot)$ is a contraction mapping, and the iteration result $\theta_t$ generated by \eqref{eq:gradient descent exact} will converge to a unique stable point $\theta_s$. 

We next account for the effect of the approximation error in Algorithm~\ref{alg:gradient descent}. 
Following the way of defining \eqref{eq:gradient estimate exact}, we denote the gradient estimate \eqref{eq:gradient estimate} as $  g (\theta_t)$. 
By the Lipschitzness of $\nabla_\theta\phi$ in $\zeta$, we have that 
\begin{align}\label{eq:gradient estimate error}
&\|\tilde{g} (\theta_t) \!-\!g (\theta_t)\|  \!\nonumber \\ 
&\le \mathbb{E}_{\xi \sim \hat{\mathbb{P}}(\theta_t)}\big[\big\|\nabla_\theta \phi\big(\theta_t,\xi,z(\theta_t,\xi)\big)\!-\! \nabla_\theta \phi\big(\theta_t,\xi,\zeta^\ast(\theta_t,\xi)\big)  \big\|\big]  \nonumber\\
&\le L_{\theta\zeta}^\phi \epsilon. 
\end{align}
Then, we have that 
\begin{align}\label{eq:approaximate contract}
&\| G_{\rm gd}^\epsilon(\theta)-\theta_s \|  \le \| G_{\rm gd}^\epsilon(\theta)-G_{\rm gd}(\theta) \|  + \|G_{\rm gd}(\theta)- \theta_s\| \nonumber \\ 
&\le \eta \|\tilde{g} (\theta)  - g (\theta )\| +\kappa_{\rm gd}\|\theta -\theta_s \| \nonumber\\
& \le \kappa_{\rm gd}\|\theta -\theta_s \| + \eta L_{\theta\zeta}^\phi\epsilon,
\end{align}
where the second inequality follows from \eqref{eq:G_gd diff}, and by substituting \eqref{eq:contration exact}. The last inequality follows from substituting \eqref{eq:gradient estimate error} into \eqref{eq:approaximate contract}.
Replacing $\theta$ with $\theta_t$ in \eqref{eq:approaximate contract}, we obtain 
\begin{equation*}
    \|\theta_{t+1}-\theta_s\| \le \kappa_{\rm gd}\|\theta_t-\theta_s\| + \eta L_{\theta\zeta}^\phi\epsilon. 
\end{equation*}
By the infinite geometric series sum, we obtain the results.
\hfill \qed 

\noindent
\textit{Proof of Lemma~\ref{lemma:L_xi_f Lipschitz}.}
For $\xi_1,\xi_2\in \Xi$, and for any fixed $\theta$, we have that 
\begin{align*} 
&\| f(\theta,\xi_1) -f(\theta,\xi_2)\| \nonumber \\ 
&=  \big\| l(\theta,\zeta^\ast(\theta,\xi_1)) - l(\theta,\zeta^\ast(\theta,\xi_2)) \nonumber \\
&\hspace{1em}- \lambda(\theta)\big(c(\xi_1,\zeta^\ast(\theta,\xi_1))-c(\xi_2,\zeta^\ast(\theta,\xi_2)) \big) \big\|\nonumber \\
& \le  L_\zeta\|\zeta^\ast(\theta,\xi_1)-\zeta^\ast(\theta,\xi_2)\|\!\nonumber\\
&\hspace{1em}+ \! \lambda_{\max}\|c(\xi_1,\zeta^\ast(\theta,\xi_1)) - c(\xi_1,\zeta^\ast(\theta,\xi_2)) \| \nonumber\\ 
&\hspace{1em}+ \lambda_{\max}\| c(\xi_1,\zeta^\ast(\theta,\xi_2))-c(\xi_2,\zeta^\ast(\theta,\xi_2))\| \nonumber \\ 
&\le  \big(L_\zeta + 2\lambda_{\max}D_{\xi}\big)\|\zeta^\ast(\theta,\xi_1)-\zeta^\ast(\theta,\xi_2)\| \nonumber \\
&\hspace{1em}+ 2\lambda_{\max}D_{\xi}\|\xi_1-\xi_2\|\nonumber \\ 
&\le L_\xi^f \|\xi_1-\xi_2\|, 
\end{align*}
where $L_\xi^f = 2\lambda_{\max}(L_\zeta+2\lambda_{\max}D_\xi)/\mu + 2\lambda_{\max}D_\xi$. The first inequality follows from the function $l(\theta,\zeta)$ is $L_\zeta$-Lipschitz in $\zeta$ and $\lambda(\theta) \le \lambda_{\max}$. The second inequality holds by $\| c(\xi,\zeta)-c(\xi,\zeta')\| \le 2\lambda_{\max}\|\zeta-\zeta'\|$ and $\| c(\xi',\zeta)-c(\xi,\zeta)\| \le 2\lambda_{\max}\|\xi-\xi'\|$, which follow the derivation of  \eqref{eq:c difference}. The last inequality follows from  $\zeta^\ast(\theta,\xi)$ is $L_{\zeta\xi}^\phi/\mu$-Lipschitz continuous in $\xi$, as in \eqref{eq:zeta(xi) Lipschitz}, and $L_{\zeta\xi}^\phi = 2\lambda_{\max}$, as in \eqref{eq:phi lip zeta_xi}.

\bibliographystyle{ieeetr}
\bibliography{references}

\begin{thebibliography}{10}

\bibitem{spall2005introduction}
J.~C. Spall, {\em Introduction to Stochastic Search and Optimization: Estimation, Simulation, and Control}.
\newblock Hoboken, New Jersey: John Wiley \& Sons, 2005.

\bibitem{tang2023zeroth}
Z.~Tang, D.~Rybin, and T.-H. Chang, ``Zeroth-order optimization meets human feedback: Provable learning via ranking oracles,'' in {\em International Conference on Learning Representations}, 2024.

\bibitem{kraft2023stochastic}
E.~Kraft, M.~Russo, D.~Keles, and V.~Bertsch, ``Stochastic optimization of trading strategies in sequential electricity markets,'' {\em European Journal of Operational Research}, vol.~308, no.~1, pp.~400--421, 2023.

\bibitem{ma2026stochastic}
H.~Ma, M.~Zeilinger, and M.~Muehlebach, ``Stochastic online optimization for cyber-physical and robotic systems,'' {\em Machine Learning}, vol.~115, no.~1, p.~11, 2026.

\bibitem{besbes2015non}
O.~Besbes, Y.~Gur, and A.~Zeevi, ``Non-stationary stochastic optimization,'' {\em Operations Research}, vol.~63, no.~5, pp.~1227--1244, 2015.

\bibitem{keimer2018information}
A.~Keimer, N.~Laurent-Brouty, F.~Farokhi, H.~Signargout, V.~Cvetkovic, A.~M. Bayen, and K.~H. Johansson, ``Information patterns in the modeling and design of mobility management services,'' {\em Proceedings of the IEEE}, vol.~106, no.~4, pp.~554--576, 2018.

\bibitem{perdomo2020performative}
J.~Perdomo, T.~Zrnic, C.~Mendler-D{\"u}nner, and M.~Hardt, ``Performative prediction,'' in {\em International Conference on Machine Learning}, pp.~7599--7609, 2020.

\bibitem{hardt2025performative}
M.~Hardt and C.~Mendler-D{\"u}nner, ``Performative prediction: Past and future,'' {\em Statistical Science}, vol.~40, no.~3, pp.~417--436, 2025.

\bibitem{bianchin2023online}
G.~Bianchin, M.~Vaquero, J.~Cortes, and E.~Dall'Anese, ``Online stochastic optimization for unknown linear systems: Data-driven controller synthesis and analysis,'' {\em IEEE Transactions on Automatic Control}, vol.~69, no.~7, pp.~4411--4426, 2023.

\bibitem{narang2023multiplayer}
A.~Narang, E.~Faulkner, D.~Drusvyatskiy, M.~Fazel, and L.~J. Ratliff, ``Multiplayer performative prediction: Learning in decision-dependent games,'' {\em Journal of Machine Learning Research}, vol.~24, no.~202, pp.~1--56, 2023.

\bibitem{wang2023network}
X.~Wang, C.-Y. Yau, and H.~T. Wai, ``Network effects in performative prediction games,'' in {\em International Conference on Machine Learning}, pp.~36514--36540, PMLR, 2023.

\bibitem{le2025learning}
H.~Le~Cadre, M.~Datar, M.~Guckert, and E.~Altman, ``Learning market equilibria preserving statistical privacy using performative prediction,'' {\em IEEE Transactions on Automatic Control}, vol.~70, no.~11, pp.~7125--7140, 2025.

\bibitem{bertrand2023stability}
Q.~Bertrand, A.~J. Bose, A.~Duplessis, M.~Jiralerspong, and G.~Gidel, ``On the stability of iterative retraining of generative models on their own data,'' in {\em International Conference on Learning Representations}, 2024.

\bibitem{miller2021outside}
J.~P. Miller, J.~C. Perdomo, and T.~Zrnic, ``Outside the echo chamber: Optimizing the performative risk,'' in {\em International Conference on Machine Learning}, pp.~7710--7720, PMLR, 2021.

\bibitem{mohajerin2018data}
P.~M. Esfahani and D.~Kuhn, ``Data-driven distributionally robust optimization using the wasserstein metric: Performance guarantees and tractable reformulations,'' {\em Mathematical Programming}, vol.~171, no.~1, pp.~115--166, 2018.

\bibitem{blanchet2019quantifying}
J.~Blanchet and K.~Murthy, ``Quantifying distributional model risk via optimal transport,'' {\em Mathematics of Operations Research}, vol.~44, no.~2, pp.~565--600, 2019.

\bibitem{kuhn2025distributionally}
D.~Kuhn, S.~Shafiee, and W.~Wiesemann, ``Distributionally robust optimization,'' {\em Acta Numerica}, vol.~34, pp.~579--804, 2025.

\bibitem{duchi2021statistics}
J.~C. Duchi, P.~W. Glynn, and H.~Namkoong, ``Statistics of robust optimization: A generalized empirical likelihood approach,'' {\em Mathematics of Operations Research}, vol.~46, no.~3, pp.~946--969, 2021.

\bibitem{shafieezadeh2019regularization}
S.~Shafieezadeh-Abadeh, D.~Kuhn, and P.~M. Esfahani, ``Regularization via mass transportation,'' {\em Journal of Machine Learning Research}, vol.~20, no.~103, pp.~1--68, 2019.

\bibitem{shafieezadeh2018wasserstein}
S.~Shafieezadeh-Abadeh, V.~A. Nguyen, D.~Kuhn, and P.~Mohajerin~Esfahani, ``Wasserstein distributionally robust {K}alman filtering,'' in {\em Advances in Neural Information Processing Systems}, 2018.

\bibitem{coulson2021distributionally}
J.~Coulson, J.~Lygeros, and F.~D{\"o}rfler, ``Distributionally robust chance constrained data-enabled predictive control,'' {\em IEEE Transactions on Automatic Control}, vol.~67, no.~7, pp.~3289--3304, 2021.

\bibitem{taskesen2023distributionally}
B.~Taskesen, D.~Iancu, {\c{C}}.~Ko{\c{c}}yi{\u{g}}it, and D.~Kuhn, ``Distributionally robust linear quadratic control,'' in {\em Advances in Neural Information Processing Systems}, 2023.

\bibitem{mcallister2024distributionally}
R.~D. McAllister and P.~M. Esfahani, ``Distributionally robust model predictive control: Closed-loop guarantees and scalable algorithms,'' {\em IEEE Transactions on Automatic Control}, 2024.

\bibitem{brouillon2025distributionally}
J.-S. Brouillon, A.~Martin, J.~Lygeros, F.~D{\"o}rfler, and G.~Ferrari-Trecate, ``Distributionally robust infinite-horizon control: from a pool of samples to the design of dependable controllers,'' {\em IEEE Transactions on Automatic Control}, 2025.

\bibitem{kuhn2019wasserstein}
D.~Kuhn, P.~M. Esfahani, V.~A. Nguyen, and S.~Shafieezadeh-Abadeh, ``Wasserstein distributionally robust optimization: Theory and applications in machine learning,'' in {\em Operations Research \& Management Science In the Age of Analytics}, pp.~130--166, Informs, 2019.

\bibitem{liu2022distributionally}
Z.~Liu, Q.~Bai, J.~Blanchet, P.~Dong, W.~Xu, Z.~Zhou, and Z.~Zhou, ``Distributionally robust {$Q$}-learning,'' in {\em International Conference on Machine Learning}, 2022.

\bibitem{luo2020distributionally}
F.~Luo and S.~Mehrotra, ``Distributionally robust optimization with decision dependent ambiguity sets,'' {\em Optimization Letters}, vol.~14, no.~8, pp.~2565--2594, 2020.

\bibitem{qu2025decision}
C.~Qu, H.~Jia, and P.~You, ``Decision-dependent distributionally robust optimization with application to dynamic pricing,'' in {\em Proc. of the 64th IEEE Conference on Decision and Control}, 2025.

\bibitem{noyan2020distributionally}
N.~Noyan, G.~Rudolf, and M.~Lejeune, ``Distributionally robust optimization with decision-dependent ambiguity set,'' {\em Optimization Letters}, vol.~14, no.~8, pp.~2541--2564, 2020.

\bibitem{fonseca2023decision}
D.~Fonseca and M.~Junca, ``Decision-dependent distributionally robust optimization,'' {\em arXiv preprint arXiv:2303.03971}, 2023.

\bibitem{xue2024distributionally}
S.~Xue and Y.~Sun, ``Distributionally robust performative prediction,'' in {\em Advances in Neural Information Processing Systems}, 2024.

\bibitem{jia2025distributionally}
Z.~Jia, Y.~Wang, R.~Dong, and G.~A. Hanasusanto, ``Distributionally robust performative optimization,'' in {\em Advances in Neural Information Processing Systems}, 2025.

\bibitem{sinha2017certifying}
A.~Sinha, H.~Namkoong, R.~Volpi, and J.~C. Duchi, ``Certifying some distributional robustness with principled adversarial training,'' in {\em International Conference on Learning Representations}, 2018.

\bibitem{edwards2011kantorovich}
D.~A. Edwards, ``On the {Kantorovich--Rubinstein} theorem,'' {\em Expositiones Mathematicae}, vol.~29, no.~4, pp.~387--398, 2011.

\bibitem{GiveMeSomeCredit}
C.~Fusion and W.~Cukierski, ``Give me some credit.'' \url{https://kaggle.com/competitions/GiveMeSomeCredit}, 2011.
\newblock Kaggle.

\end{thebibliography}

\end{document}